\documentclass{amsart}
\usepackage{amssymb,amsmath,amsthm}
\usepackage[utf8]{inputenc}

\usepackage{color}
\usepackage[all]{xy}
\date{\today}

\newcommand{\smallk}{{\mathbf{k}}}

\renewcommand{\preceq}{\preccurlyeq}
\newcommand{\supp}{\operatorname{supp}}
\newcommand{\truncof}{\triangleleft}
\newcommand{\trunceq}{\trianglelefteq}

\newcommand{\til}{\operatorname{TIL}}
\newcommand{\il}{\operatorname{IL}}
\newcommand{\Texp}{\mathbb{T}_{\exp}}

\newcommand{\cF}{\mathcal{F}}

\newcommand{\R}{\mathbb{R}}
\newcommand{\T}{\mathbb{T}}

\newcommand{\fg}{\mathfrak{g}}
\newcommand{\fm}{\mathfrak{m}}
\newcommand{\fn}{\mathfrak{n}}
\newcommand{\fp}{\mathfrak{p}}
\newcommand{\fq}{\mathfrak{q}}

\newcommand{\fG}{\mathfrak{G}}
\newcommand{\fM}{\mathfrak{M}}
\newcommand{\fN}{\mathfrak{N}}
\newcommand{\fP}{\mathfrak{P}}
\newcommand{\fQ}{\mathfrak{Q}}

\newcommand{\monm}{\mathfrak{m}}
\newcommand{\bigO}{\mathcal{O}}
\newcommand{\smallo}{{\scriptstyle \mathcal{O} }}
\newcommand{\ser}[1]{\sum_{\gamma} #1_\gamma t^\gamma}

\newcommand{\upper}{\hspace{-0.12cm}\uparrow}
\newcommand{\downer}{\hspace{-0.15cm}\downarrow}

\newcommand{\inv}{^{-1}}

\DeclareSymbolFont{der@m}{U}{fsy}{m}{n}
\DeclareMathSymbol{\derdelta}{\mathord}{der@m}{100}

\def \>{\rangle}

\newtheorem{theorem}{Theorem}[section]
\newtheorem{lemma}[theorem]{Lemma}
\newtheorem{corollary}[theorem]{Corollary}
\newtheorem{prop}[theorem]{Proposition}

\title{Truncation In Unions of Hahn Fields with a Derivation}
\author{Santiago Camacho}
\email{scamach2@illinois.edu}
\address{Department of Mathematics, University of Illinois at Urbana-Champaign, Urbana, Illinois 61801}

\begin{document}

\begin{abstract}
Truncation in Generalized Series fields is a robust notion, in the sense that it is preserved under various algebraic and some transcendental extensions. In this paper we study conditions that ensure that a truncation closed set extends naturally to a truncation closed differential ring, and a truncation closed differential field has a truncation closed Liouville closure. In particular, we introduce the Notion of IL-closedness in Unions of Hahn fields in order to determine that this condition is sufficient to preserve truncation in those two settings for constructions such as the field of logarithmic-exponential transseries $\mathbb{T}$.
\end{abstract}

\maketitle

\setcounter{tocdepth}{1}
\tableofcontents
\section{Introduction}

Mourgues and Ressayre~\cite{MR} showed that any real closed field is 
isomorphic to a
truncation closed subfield of a Hahn field over $\R$. 
The papers \cite{DMM, F, FKK, D} continue the study of
truncation closedness. The results 
are typically that truncation closedness is robust in the sense of being 
preserved under a variety of extensions. One significance of 
truncation closedness is that it enables transfinite induction to be imported as
a tool into valuation theory.
 

 There is increasing interest in Hahn fields with a `good'
derivation, and truncation closedness is potentially significant 
in that context for similar reasons. We show here that 
truncation closedness is preserved under certain extensions that involve
the derivation.  Our main goal is to establish results for the differential field  $\mathbb{T}$ of transseries in the sense of \cite{ADH}, but initially
we work in a simpler and rather general setting of Hahn fields with a 
`good' derivation. To apply results in that setting to $\T$ we use that
$\T$ is, roughly speaking, obtained by iterating a 
Hahn field construction: at each step one builds a Hahn field-with-derivation
on top of a previously constructed Hahn field-with-derivation. 

For subsets of $\T$ we introduce the condition 
of being \textbf{iteratively-logarithmically} closed, $\il$-closed for short. 
We prove two preservation results for truncation closed 
$\il$-closed subsets of $\mathbb{T}$ , $\til$-closed for short. 

\begin{theorem}
If $K$ is a $\til$-closed subfield of $\mathbb{T}$ containing $\R$, 
then the differential subfield of $\mathbb{T}$ generated by $K$ is $\til$-closed.
\end{theorem}

\noindent
Following Aschenbrenner and van den Dries \cite{AD}, an \textbf{$H$-field}
is an ordered differential field $K$ with field of constants $C$ such that
\begin{itemize}
\item[i)] $f\in K$, $f>C$ $\Rightarrow$ $f'>0$,
\item[ii)] $\bigO = C + \smallo$ where $\bigO :=\{f\in K: |f|<|c| \text{ for some } c\in C \}$, and $\smallo$ is the maximal ideal of the convex subring $\bigO$ of $K$.  
\end{itemize}

\noindent
It is well known that any ordered differential subfield of $\mathbb{T}$ that contains $\R$ is again an $H$-field.
An $H$-field $K$ is said to be \textbf{Liouville closed} if $K$ is real closed, and for each $f\in K^\times$ there exist $g,h\in K^\times$ such that 
$g'=f$ and $h^\dagger : = h'/h = f$. In the case of $H$-subfields of 
$\mathbb{T}$ containing $\R$ an equivalent condition is being real closed, closed under integration, and closed under exponentiation. Our second result regards diferential extensions of $H$-fields that are Liouville closed.  

\begin{theorem}
Let $K$ be a $\til$-closed differential subfield of $\mathbb{T}$ containing $\R$. Then the smallest Liouville closed differential subfield of $\mathbb{T}$ containing $K$  is also $\til$-closed. 
\end{theorem}

\section{Notations and Preliminaries}

\noindent
We let $m,n$ range over the set $\mathbb{N}=\{0,1,2,\dots\}$ of natural numbers.
By convention, ordered sets (and ordered abelian groups, ordered fields) are totally ordered. For an ordered set $S$ and $a\in S$ we set $S^{>a}=\{s\in S: s > a\}$ and similarly for $<,\leq, \geq, \neq$ in place of $>$.
 We let
$\fM$ denote a multiplicative ordered abelian group, whose elements $\fm$ are 
thought of as monomials; the
(strict) ordering on $\fM$ is denoted by $\prec$ (or $\prec_{\fM}$ if we need to
indicate $\fM$); likewise with $\fN$. Sometimes it is more convenient to use
additive notation, and so we let
$\Gamma$ denote an additive ordered abelian group (with zero element $0_\Gamma$ if we need to indicate the dependence on $\Gamma$); then $<$ rather than $\prec$
denotes the strict ordering of $\Gamma$; also, $\Gamma^>:=\Gamma^{>0}$,
and likewise for $<,\leq,\geq$ and $\neq$.  When $\Gamma$ is clear from the context, we let $\alpha, \beta, \gamma$ range over $\Gamma$. By ``ring'' we mean a 
commutative ring with $1$. Throughout, 
$\smallk$ is a field.

\subsection*{Hahn Series} The elements of the \textbf{Hahn field} $\smallk[[\fM]]$ are the formal series $f=\sum_{\fm} f_{\fm}\fm$ with coefficients 
$f_\fm\in \smallk$ and monomials $\fm\in \fM$ such that 
$$\supp f\ :=\ \{\fm:\ f_{\fm}\ne 0\}$$ is anti-well-ordered, that is,
$\supp f$ contains no strictly increasing
infinite sequence $\fm_0\prec \fm_1 \prec \fm_2 \prec\cdots$.
These series are added and multiplied in the way suggested by the series
representation, and this makes $\smallk[[\fM]]$ into a field with 
$\smallk$ as a subfield via the identification of $a\in \smallk$ with 
the series  $f=\sum f_{\fm}\fm$ such that $f_{1}=a$ and  $f_{\fm}=0$ for all
$\fm\ne 1$. 



\medskip\noindent
Often $\fM=t^\Gamma$ where $t$ is just a symbol, and
$\gamma\mapsto t^\gamma: \Gamma\to t^\Gamma=\fM$ is an {\em order-reversing\/}
isomorphism of $\Gamma$ onto $\fM$. Then we denote $\smallk[[\fM]]$
also by $\smallk((t^\Gamma))$, and write 
$f=\sum_{\fm} f_{\fm}\fm\in \smallk[[\fM]]$ as
$f=\sum_\gamma f_{\gamma}t^\gamma$, with $f_\gamma:= f_{\fm}$ for $\fm=t^\gamma$.
In this situation we prefer to take $\supp f$ as a subset of $\Gamma$ 
rather than of $\fM=t^\Gamma$: $\supp f =\{\gamma:\ f_{\gamma}\ne 0\}$, and the
anti-well-ordered requirement turns into the requirement that
$\supp f$ is a well-ordered subset of $\Gamma$. Of course, all this is only 
a matter of notation, and we shall freely apply results for
Hahn fields $\smallk((t^\Gamma))$ to Hahn fields $\smallk[[\fM]]$, since we 
can always pretend that $\fM$ is of the form $t^\Gamma$.

The Hahn field $\smallk((t^\Gamma))$ comes equipped with the (Krull) valuation 
$v:=v_\Gamma:\smallk((t^\Gamma))^\times\rightarrow \Gamma$ given by
$v(f)=\min \supp f$. Given any (Krull) valued field $K$ with valuation $v$
we have the binary (asymptotic) relations on $K$ given by 
\begin{itemize}
	\item $f\preceq g :\iff v(f)>v(g)$
    \item $f\prec g :\iff (f\preceq g\ \&\ g\not\preceq f)$
    \item $f\asymp g :\iff (f\preceq g\ \&\ g\preceq f)$
\end{itemize}
We write $\preceq_v$ or $\preceq_\Gamma$ if we need to specify the valuation,
or the value group $\Gamma$ in the case of the valued Hahn field 
$K=\smallk((t^\Gamma))$. 
Given a subset $S$ of a valued field $K$ and $f\in K$ we set $S^{\prec f}:=\{g\in K:\ g\prec f\}$. 


\subsection*{Differential rings} Let $R$ be a ring.  A \textbf{derivation} on 
$R$ is an additive map  $\partial: R \to R$ that satisfies the Leibniz rule: for all $a,b\in R$, 
\begin{itemize}
	\item $\partial(a+b) = \partial(a) + \partial(b)$
    \item $\partial(ab) = \partial(a)b + a\partial(b)$.
\end{itemize}
A \textbf{differential ring} is a ring together with a derivation on it.
Let $R$ be a differential ring. Unless we specify otherwise, we let
$\partial$ be the derivation of $R$, and for $a\in R$ we let $a'$ denote 
$\partial(a)$. The 
\textbf{ ring of constants of $R$} is the subring $C_R:=\{a\in R:\ a'=0\}$ of $R$. We also have the ring $R\{Y\}$ of differential
polynomials in the indeterminate $Y$ over $R$: as a ring, this is just the
polynomial ring $$R[Y, Y', Y'',\dots]\ =\ R[Y^{(n)}:n=0,1,2,\dots]$$ in the distinct 
indeterminates $Y^{(n)}$ over $R$, and it is made into a differential ring extension of $R$ by requiring that $(Y^{(n)})'= Y^{(n+1)}$ for all $n$.

For a differential field $K$ and $f\in K^\times$, we let
 $f^\dagger$ denote the logarithmic derivative $\partial(f)/f$, so
$(fg)^\dagger=f^\dagger + g^\dagger$ for $f,g\in K^\times$, and 
for a subset $S$ of $K^\times$ we
set $S^\dagger:=\{f^\dagger:\ f\in S\}$. 

\section{Small and Strong Additive Operators}
\noindent
Let $C$ be an additive abelian group, and $\fM$ a totally ordered monomial set.
We consider the Hahn space $C[[\fM]]$ of elements $f= \sum_{\fm\in\fM}f_{\fm}\fm$ with reverse well-ordered (or Noetherian) support. 
\medskip
\noindent
Let $(f_i)_{i\in I}$ be a family in $C[[\fM]]$. We say that this family is \textbf{summable} if 
\begin{itemize}
	\item $\bigcup_i \supp(f_i)$ is reverse well-ordered, and
    \item for each $\fm\in \fM$ there are finitely many $i$ such that $\fm\in \supp(f_i)$.
\end{itemize}
If $(f_i)_{i\in I}$ is summable, we define its sum $f=\sum_{i\in I} f_i$ by $f_\fm = \sum_i f_{i,\fm}$. Note that for $\fG$ a reverse well-ordered subset of $\fM$ and $(f_\fg)_{\fg\in \fG}$ a family in $C$, the notation $\sum_\fg f_\fg \fg$ coincides for the series and the sum of the family $(f_\fg\fg)_\fg$.

\medskip
\noindent
An \textbf{operator} on $C[[\fM]]$ is by definition a map $P: C[[\fM]]\rightarrow C[[\fM]]$. 
We call an operator $P:C[[\fM]]\rightarrow C[[\fM]]$ \textbf{additive} if $P(f+g) = P(f)+P(g)$ for all $f,g\in C[[\fM]]$. 
The set of all additive operators on $C[[\fM]]$ with the above addition and multiplication given by composition is a ring with the null operator $O$ as its zero element and the identity operator $I$ on $C[[\fM]]$ as its identity element. 

\subsection*{Strong Operators}
We call an operator $P$ \textbf{strongly additive}, or \textbf{strong} for short, if for every summable family $(f_i)$ in $C[[\fM]]$ the family $(P(f_i))$ is summable and $\sum P(f_i) = P(\sum f_i)$. Strong operators on $C[[\fM]]$ are additive, and the null operator $O$ and the identity operator $I$ are strong. The following is a routine consequence of the definition of ``strong''. 

\begin{lemma}
If $P,Q$ are strong operators on $C[[\fM]]$, then so are $P+Q$ and $PQ$.
\end{lemma}
\begin{proof}
Let $(f_i)_{i\in I}$ be a summable family in $\smallk[[\fM]]$. We have \[\bigcup_i \supp(P(f_i) + Q(f_i)) \subseteq \left( \bigcup_i \supp P(f_i) \right)\cup \left( \bigcup_i \supp Q(f_i) \right),\]
Which is reverse well ordered. and for each $\fn \in \fN$ we have 
\[\{i\in I: (P+Q)(f_i)_{\fn} \neq 0 \} \subseteq \{i\in I: P(f_i)_{\fn} \neq 0 \} \cup \{i\in I: Q(f_i)_{\fn} \neq 0 \},\]
which is finite, thus proving (1). For (2) $(Q(f_i))$ is summable by assumption, so $(P(Q(f_i))) = (P Q(f_i))$ is summable and 
\[\sum P(Q(f_i)) = P\left(\sum Q(f_i) \right)= PQ\left(\sum f_i\right).\qedhere\]
\end{proof}

\medskip\noindent
Thus the family of strong operators on $C[[\fM]]$ is a subring of the ring of additive operators on $C[[\fM]]$. Let $(P_i)_{i\in I}$ be a family of additive operators on $C[[\fM]]$. We say that the family $(P_i)$ of additive operators is \textbf{summable} if for all $f\in C[[\fM]]$ the family $(P_i(f))_{i\in I}$ is summable, and if in addition the map $f\mapsto \sum_i P_i(f)$ is strongly additive we say that the family $(P_i)$ is \textbf{strongly summable}.
We denote the map $f\mapsto \sum_i P_i(f)$ as $\sum_{i\in I} P_i$.
In any case, if $(P_i)$ is summable, then $\sum P_i$ is additive.

\begin{prop}\label{PropsExist}
 Let $(P_i)_{i\in I}$ be a family of additive operators on $C[[\fM]]$.
  \begin{enumerate}
  \item If I is finite, then $(P_i)_{i\in I}$ is summable.
    \item If $I$ and $J$ are disjoint sets, $(P_i)_{i\in I}$ is summable, and $(P_j)_{j\in J}$ is also a family of summable additive operators, then the family $(P_i)_{i\in I\cup J}$ is summable. 
    \item If $\sigma : I \rightarrow J$ is a bijection, and $(P_i)_{i\in I}$ is summable, then $(P_{\sigma\inv (j)})_{j\in J}$ is summable.
    \item Let $I = \bigcup_{j\in J} I_j$ with pairwise disjoint $I_j$. If $(P_i)_{i\in I}$ is summable, then $(P_i)_{i\in I_j}$ is summable for each $j\in J$ and $\sum_{j\in J}\sum_{i\in I_j}P_i = \sum_{i \in I }P_i$
\end{enumerate}
\end{prop}

\begin{prop}
Let $(P_i)_{i\in I}$ be a family of strong operators such that $(P_i)_{i\in I}$ is strongly summable and let $Q$ be a strong operator. Then $(P_iQ)_{i\in I}$ and $(QP_i)_{i\in I}$ are strongly summable  and equal to $(\sum_i P_i)Q$ and $Q\sum_iP_i$, respectively.
\end{prop}

\subsection*{Small Operators}
\noindent
In this section $\fM$ is a commutative monomial group and $C$ a commutative ring. 
We make the ring of strong operators on $C[[\fM]]$ into a $C$-algebra by defining $(cP)(f):= c\cdot P(f)$ for $c\in C,P$ a strong operator on $C[[\fM]]$, and $f\in C[[\fM]]$. 

\noindent
An operator $P$ on $C[[\fM]]$ is said to be \textbf{small}\footnote{Small operators were defined in \cite[p. 66]{DMM} without requiring additivity, but this invalidates the assertion there about the inverse of $I-P$. Fortunately, this assertion is only used later in that paper for additive $P$, for which it is correct.} , if $P$ is additive and there is a reverse well-ordered $\fG\subseteq \fM^{\prec 1}$ such that $\supp(P(f)) \subseteq \fG\supp(f)$ for every $f\in C[[\fM]]$; any such $\fG$ is called a \textbf{witness} for $P$. 

\begin{prop}\label{fGfN} If $\fG$ and $\fN$ are reverse well-ordered subsets of $\fM$, then $\fG\fN$ is reverse well-ordered, and for every $\fm\in \fM$ there are finitely many pairs $(\fg,\fn)\in \fG\times\fN$ such that $\fg\fn=\fm$.
\end{prop}
\begin{proof}
Let $\fg_1\fn_1 \prec \fg_2\fn_2 \prec \cdots$ be an increasing family with $\fg_i \in \fG$ and $\fn_i\in \fN$ by reverse well-orderedness of $\fG$ there is a subsequence $\fg_{i_1} \succeq \fg_{1_2} \succeq\cdots$  of $\fg_1, \fg_2 , \ldots $ which is decreasing, but that would mean that the sequence $\fn_{1_2},\fn_{i_2}, \ldots$ is strictly increasing, a contradiction. 
\end{proof}

\noindent
Small operators on $C[[\fM]]$ are contained in the set of strong operators on $C[[\fM]]$:

\begin{lemma}
If $P$ is a small operator on $C[[\fM]]$, then $P$ is strong.
\end{lemma}

\begin{proof}
Let $(f_i)_{i\in I}$ be a summable family. 
\[\bigcup_i \supp(P(f_i)) \subseteq \fG\bigcup_i \supp(f_i),\]
which is an reverse well-ordered subset of $\fM$. Let $\fm \in \fM$. Since there are only finitely many pairs $(\fg,\fn)\in \fG\times \bigcup_i \supp(f_i)$ and for each $\fn$ there are only finitely many $i$ such that $\fn\in \supp(f_i)$, there are only finitely many $i$ such that $\fm \in \supp(P(f_i))$. Thus $P$ is a strong operator. 
\end{proof}

\noindent
In fact the small operators on $C[[\fM]]$ form a subring of the strong operators on $C[[\fM]]$
\begin{lemma}
Let $P,Q$ be small operators, then so are $P+Q$ and $P Q$.
\end{lemma}
\begin{proof}
Assume $P$ and $Q$ are small and let $\fG$ be a witness for $P$ and $\fN$ a witness for $Q$, then $\fG\cup \fN$ and $\fG \fN$ are witnesses for $P+Q$ and $P Q$ respectively. 
\end{proof}

\noindent
The following proposition is useful for constructing inverses of certain operators, the statement appears originally in \cite{N} and is usually refered to as Neumann's Lemma. A proof using  additive notation appears in \cite[Chapter 5]{G}.
\begin{prop}\label{fGstar} Let $\fG$ be an reverse well-ordered subset of $\fM^{\prec 1}$. Then $\fG^*:=\bigcup_n \fG^n$ is reverse well-ordered. Moreover, for any $\fm\in \fM$ there are only finitely many tuples $(n,\fg_1, \ldots ,\fg_n)$ such that $\fg_1, \ldots \fg_n \in \fG$ and $\fg_1\cdots \fg_n = \fm$.
\end{prop}

\begin{prop}Let $(P_n)$ be a family of small operators such that there is $\fP\subseteq \fM^{\prec 1}$ with $\fP^n$ a witness for $P_n$. Then $(P_n)$ is strongly summable.
\end{prop}

\begin{prop}\label{cnPn}
If $(c_n)$ is a sequence of elements in $C$ and $P$ a small operator. Then $ (c_nP^n)$ is strongly summable.
\end{prop}

\begin{prop} Let $P$ be a small operator, then $I-P:C[[\fM]]\rightarrow C[[\fM]]$ is bijective with inverse given by $\sum_n P^n$. Moreover $(I-P)$ and its inverse are strong. 
\end{prop}
\begin{proof} The existence of the inverse for $(I-P)$ follows from Proposition \ref{cnPn} by taking $c_n = 1$ for all $n$. We only prove strong additivity of the inverse $Q= (I-P)\inv$. Let $(f_i)_i$ be a summable family in $\smallk[[\fM]]$. We claim that the family $(Q(f_i))_i$ is summable. For every $f\in \smallk[[\fM]]$ the support of $Q(f)$ is contained in $\fG^*\supp(f)$. Thus 
\[\bigcup_i \supp(Q(f_i))\subseteq \bigcup_i \fG^*\supp(f_i) = \fG^*\bigcup_i \supp(f_i),\]
which is reverse well-ordered. Now let $\fn\in \fM$. If $\fn\in \supp(Q(f_i))$ then there are $g_i\in \fG^*$ and $\fm_i\in \supp(f_i)$ such that $g_i\fm_i= \fn$. There are only finitely many pairs $(g,\fm)\in \fG^*\times \bigcup_i(\supp(f_i))$ such that $g\fm= \fn$. Since every $\fm$ appears only in finitely many of the sets $\supp(f_i)$ we conclude that $\fn$ appears in the support of $Q(f_i)$ for only finitely many $i\in I$. Thus $(Q(f_i))_i$ is a summable family and it is easy to see that its sum is $Q(\sum_if_i)$.
\end{proof}


\noindent
Note that $Q(I-P)\inv$ is a small operator with witness $\fN\fG^*$.

\begin{lemma}Let $(Q_n)_{n\geq 1}$ be a family of small operators and $\fQ\subseteq \fM^{\prec 1}$ such that $\fQ^n$ is a witness for $Q_n$ and let $(P_i)_{i\in I}$ be a family of small operators that share a common witness $\fP$. If both families $(Q_n)$ and $(P_i)_{i\in I}$ are strongly summable, then $(P_iQ_n)_{i,n}$ is strongly summable and $\sum_{i,n} P_iQ_n = \sum_i P_i \sum_n Q_n$. 
\end{lemma}

\begin{proof}
First we note that for any $f$, the family $(P_iQ_n(f))$ is summable. Let $\mathcal{F}$ denote the support of $f$. Then
\[\bigcup_{i,n} \supp(P_iQ_n(f))\subseteq \bigcup_n \fQ^n\fP \cF \subseteq \fQ^*\fP\cF,\]
which is anti-well-ordered. 
We now let $\fm$ be a fixed element in $\fM$. 
We want to show that $\{(i,n): \fm\in \supp(P_iQ_n(f))\}$ is finite. 
Since $\sum_i P_i$ exists, we have that for a fixed $n$, $\{i: \fm\in \supp P_iQ_n(f)\}$ is finite. Now, by combining \ref{fGstar} and \ref{fGfN} we know that there are only finitely many tuples $(\fp,(n,\fq_1,\ldots,\fq_n),\fn)$ with $\fp\in \fP, n\in \mathbb{N}, \fq_i\in \fQ, \fn\in \cF$ such that $\fp\fq_1\cdots\fq_n\fn = \fm$. These two observations show that $(P_iQ_n)$ is summable. The existence then follows from smallness, and equality from Proposition \ref{PropsExist}.
\end{proof}

\section{Truncation of Hahn Series} 

\noindent
For $f=\sum_{\fm}f_{\fm}\fm\in \smallk[[\fM]]$
and $\fn\in \fM$, the \textbf{truncation  $f|_{\fn}$ of $f$ at $\fn$} is defined by
$$ f|_{\fn}\ :=\ \sum_{\fm\succ \fn}f_{\fm}\fm, \quad\text{an element of } \smallk[[\fM].$$
Thus for $f,g\in \smallk[[\fM]]$ we have $(f+g)|_{\fn} = f|_{_\fn} +g|_{\fn}$. 
A subset $S$ of $\smallk[[\fM]]$ is said to be \textbf{truncation closed} if for all $f\in S$ and $\fn\in \fM$ we have $f|_{\fn} \in S$. For 
$f,g\in \smallk[[\fM]]$ we let $f\trunceq g$ mean that $f$ is a truncation of 
$g$, and let $f\truncof g$ mean that $f$ is a proper truncation of $g$,
that is, $f\trunceq g$ and $f\ne g$. 

\medskip\noindent
When, as in \cite{D}, Hahn fields are given as $\smallk((t^\Gamma))$, 
we adapt our notation accordingly:
the \textbf{truncation of $f =\ser{f}\in \smallk((t^\Gamma))$ at $\gamma_0\in \Gamma$} is 
\[f|_{\gamma_0}\ :=\ \sum_{\gamma<\gamma_0}f_\gamma t^\gamma.\]
We now list some items from \cite{D} that we are going to use:
 
\begin{prop} \label{D} Let $A$ be a subset of $\smallk[[\fM]]$, $R$ a subring 
of $\smallk[[\fM]]$, and $E$ a (valued) subfield of $\smallk[[\fM]]$. Then:
\begin{itemize}
	\item[i)] The ring as well as the field generated in $\smallk[[\fM]]$ by any truncation closed
subset of $\smallk[[\fM]]$ is truncation closed: \cite[Theorem 1.1]{D};
    \item[ii)] If $R$ is truncation closed and all proper truncations of all $a\in A$ lie in $R[A]$, then the ring $R[A]$ is truncation closed;
    \item[iii)] If $E$ is truncation closed and $E\supseteq \smallk$, 
then the henselization of $E$ in 
$\smallk[[\fM]]$ is truncation closed: \cite[Theorem 1.2]{D};
    \item[iv)] If $E$ is truncation closed, henselian, and $\operatorname{char}(\smallk)=0$, then any algebraic field extension of $E$ in $\smallk[[\fM]]$ is 
truncation closed: \cite[Theorem 5.1]{D}. 
\end{itemize}
\end{prop}

\noindent
Item (ii) here is an improved version of \cite[Corollary 2.6]{D}.
To prove (ii), assume $R$ is truncation closed and all proper truncations of
all $a\in A$ lie in $R[A]$. Let 
$$B\ :=\ A \cup \{f\in \smallk[[\fM]]:\ f\truncof a \text{ for some }a\in A\}.$$
Then $B$ is truncation closed, and so is $R\cup B$, and thus the ring
$R[B]=R[A]$ is truncation closed, by (i) of the proposition above.

\medskip\noindent
The above deals with extension procedures of algebraic nature. 
As in \cite{D}, we also consider transcendental adjunctions of
of the following kind.
Let for each $n\ge 1$ a subset
$\cF_n$ of $\smallk[[X_1,\dots, X_n]]$ be given such that the subring
$\smallk[X_1,\dots, X_n, \cF_n]$ of $\smallk[[X_1,\dots, X_n]]$ is closed under
$\partial/\partial X_i$ for $i=1,\dots,n$, and let $\cF$ be the family
$(\cF_n)$. For example, if $\text{char}(\smallk)=0$, we could take
$$\cF_1=\{(1+X_1)^{-1},\ \exp X_1,\ \log(1+X_1)\},\qquad  \cF_n=\emptyset\  \text{ for }n>1$$
where $\exp X_1:=\sum_{i=0}^\infty X_1^i/i!\ $  and $\log(1+X_1):=\sum_{i=1}^{\infty}(-1)^{i+1}X_1^i/i$.
A subfield $E$ of $\smallk[[\fM]]$ is said to be $\cF$-closed if 
$f(\vec a)\in E$ for all
$f(X_1,\dots, X_n)\in \cF_n$ and $\vec a\in (E\cap\ \smallo )^{\times n}$, $n=1,2,\dots$. The $\cF$-closure of a subfield
$E$ of $\smallk[[\fM]]$ is the smallest $\cF$-closed subfield 
$\cF(E)$ of $\smallk[[\fM]]$ that contains $E$.

\begin{prop}\label{ta} If $\operatorname{char}(\smallk)=0$ and
$E\supseteq \smallk$ is a truncation closed subfield of $\smallk[[\fM]]$,
then its $\cF$-closure $\cF(E)$ is also truncation closed.
\end{prop}

\subsection*{Additional facts on truncation} 
Besides $\Gamma$ we now consider a second ordered abelian 
group $\Delta$. Below we identify $\Gamma$ and $\Delta$ 
in the usual way with subgroups of
the lexicographically ordered sum $\Gamma\oplus \Delta$, so that 
$\Gamma+\Delta=\Gamma\oplus \Delta$, with $\gamma>\Delta$ for all $\gamma\in \Gamma^{>}$. This makes $\Delta$ a convex subgroup of $\Gamma+\Delta$. 
Let $\smallk_0$ be a field and $\smallk = \smallk_0((t^\Delta))$.
Then we have a field isomorphism 
$$\smallk((t^\Gamma))\longrightarrow \smallk_0((t^{\Gamma+ \Delta}))$$
that is the identity on $\smallk$, namely 
$$f=\sum_{\gamma}f_{\gamma}t^\gamma\mapsto 
\sum_{\gamma,\delta} f_{\gamma,\delta}t^{\gamma+\delta}$$
where $f_{\gamma}=\sum_\delta f_{\gamma,\delta}t^{\delta}$ for all $\gamma$. 
Below we identify $\smallk((t^\Gamma))$ with  $\smallk_0((t^{\Gamma+ \Delta}))$
via the above isomorphism. 
 For a set $S\subseteq \smallk((t^\Gamma))$
this leads to two notions of truncation: 
we say that $S$ is
$\smallk$-truncation closed if it is truncation closed with 
$\smallk$ viewed as the coefficient 
field and $\Gamma$ as the group of exponents (that is, viewing $S$ as a subset of the Hahn field $\smallk((t^\Gamma))$ over $\smallk$), and we say that
$S$ is $\smallk_0$-truncation closed if it is truncation closed  with 
$\smallk_0$ viewed as the coefficient 
field and $\Gamma+\Delta$ as the group of exponents (that is, viewing
$S$ as a subset of the Hahn field $\smallk_0((t^{\Gamma+\Delta}))$
over $\smallk_0$). 

\begin{lemma}\label{trtr1} If $S\subseteq\smallk((t^\Gamma))$ is $\smallk_0$-truncation closed, then $S$ is $\smallk$-truncation closed.
\end{lemma}
\begin{proof} Let $f\in \smallk((t^\Gamma))$ and $\gamma\in \Gamma$. If 
$ (\gamma+\Delta^{<})\cap \supp_{\smallk_0}f=\emptyset$, then the
$\smallk$-truncation of $f$ at $\gamma$ equals the $\smallk_0$-truncation 
of $f$ at 
$\gamma\in \Gamma +\Delta$.  

If $(\gamma+\Delta^{<})\cap \supp_{\smallk_0}f\ne
\emptyset$, then the $\smallk$-truncation of $f$ at $\gamma$ equals the $\smallk_0$-truncation of $f$ at the least element of $(\gamma+\Delta^{<})\cap \supp_{\smallk_0}f$.
\end{proof}

\begin{lemma}\label{trtr2} Let $\smallk_1$ be a truncation closed subfield of the Hahn field 
$\smallk_0((t^\Delta))= \smallk$ over $\smallk_0$, and let $V$ be a 
$\smallk_1$-linear subspace of $\smallk_1((t^\Gamma))\subseteq \smallk((t^\Gamma))$
such that $V\supseteq \smallk_1$ and $V$ is $\smallk$-truncation closed.
Then $V$ is $\smallk_0$-truncation closed.
\end{lemma} 
\begin{proof} If $\Gamma=\{0\}$, then $\smallk_1((t^\Gamma))=\smallk_1$, so
$V=\smallk_1$  is $\smallk_0$-truncation closed. In the rest of the proof
we assume $\Gamma\ne\{0\}$. Let $\beta=\gamma+\delta\in \Gamma+\Delta$ with
$\gamma\in \Gamma$ and $\delta\in \Delta$, and $f\in V$.
Let $g$ be the truncation of $f$ at $\beta$ in the Hahn field
$\smallk_0((t^{\Gamma+\Delta}))$ over $\smallk_0$ and $h$ the truncation of $f$ at $\gamma$ in the Hahn field $\smallk((t^\Gamma))$. Then $h\in V$
and $g=h+s$ with $s=\phi t^\gamma$ and $\phi\in \smallk_0((t^\Delta))$.
If $s=0$, then $g\in V$ trivially, so assume $s\ne 0$.
Then $f$ has a $\smallk$-truncation $h+\theta t^\gamma\in V$ with 
$0\ne\theta\in \smallk_1$ and so $\theta t^\gamma\in V$, $t^\gamma=\theta^{-1}(\theta t^\gamma)\in V$. Moreover,
$\phi$ is a $\smallk_0$-truncation of $\theta$, so $\phi\in \smallk_1$, hence
$s=\phi t^\gamma\in V$, and thus $g\in V$.
\end{proof}

\section{Derivations on $\smallk((t^\Gamma))$}

\noindent
In this section $\smallk$ is a {\em differential\/} field, with derivation $\partial$ (possibly trivial). We also let $\alpha,\beta, \gamma$ range over $\Gamma$, and fix an additive map $c:\Gamma \rightarrow \smallk$.
This allows us to extend $\partial$ to a derivation on the field $\smallk((t^\Gamma))$, also denoted by $\partial$, by declaring for $f=\sum_{\gamma}f_\gamma t^\gamma\in \smallk((t^\Gamma))$ that
\[\partial \left(f\right)\ :=\  \sum_{\gamma} \Big(\partial (f_\gamma) +f_\gamma c(\gamma)\Big)t^\gamma.\]
Thus $(t^\gamma)'=c(\gamma)t^\gamma$. 
For $f\in \smallk((t^\Gamma))$ we have $\supp f'\subseteq \supp f$, so $f'\preceq f$. If
$c(\Gamma)=\{0\}$, then the constant field of $\smallk((t^\Gamma))$ is clearly
$C_{\smallk}((t^\Gamma))$. It is easy to check that the constant field of
$\smallk((t^\Gamma))$ is the same as the constant field of $\smallk$ if and only if $c$ is injective and $c(\Gamma) \cap \smallk^\dagger = \{0\}$. 

Sometimes it is more natural to consider a Hahn field $\smallk[[\fM]]$, and then
a map $c: \fM\to \smallk$ is said to be addditive if $c(\fm\fn)=c(\fm) + c(\fn)$
for all $\fm, \fn\in \fM$. Again, such $c$ allows us to extend $\partial$ to a derivation $\partial$ on the field $\smallk[[\fM]]$ by 
\[\partial \left(f\right)\ :=\  \sum_{\fm} \Big(\partial (f_\fm) +f_\fm c(\fm)\Big)\fm.\]

 \medskip\noindent
{\bf Examples.} For $\fM=x^{\mathbb{Z}}$ with $x\succ 1$ we have the usual derivation 
$\frac{d}{dx}$ with 
respect to $x$ on the field of Laurent series 
$\smallk[[x^{\mathbb{Z}}]]=\smallk((t^{\mathbb{Z}}))$ (with $t=x^{-1}$), but it is not of the form considered above.
The derivation $x\frac{d}{dx}$, however, does have the form above, with the
trivial derivation on $\smallk$ and $c(x^k)=k\cdot 1\in \smallk$ for 
$k\in \mathbb{Z}$. Likewise for $\fM=x^{\mathbb{Q}}$ ($x\succ 1$), 
and $\operatorname{char}(\smallk)=0$: then  $x\frac{d}{dx}$ has the above 
form, with the trivial derivation on $\smallk$ and 
$c(x^q)=q\cdot 1\in \smallk$ for 
$q\in \mathbb{Q}$.

\bigskip\noindent
We now return to the setting of  $\smallk((t^\Gamma))$, and observe that
$\partial(f|_\gamma) = \partial(f)|_\gamma$ for $f\in \smallk((t^\Gamma))$. Thus by Proposition \ref{D}:

\begin{corollary}\label{lem1}
If $R$ be a truncation closed subring of $\smallk((t^\Gamma))$ and $f\in R$ is such that $\partial(g)\in F$ for every proper truncation $g$ of $f$, 
then all proper truncations of $\partial(f)$ lie in $R$ and thus 
$R[\partial(f)]$ is truncation closed.
\end{corollary}

\begin{lemma}\label{DifFieldGenTrunc}
Let $R$ be a truncation closed subring of $\smallk((t^\Gamma))$. 
Then the differential subring of $\smallk((t^\Gamma))$ generated by $R$ is 
truncation closed. 
\end{lemma}

\begin{proof}
Let $R_0 = R$ and $R_{n+1} = R_n[\partial(f): f\in R_n]$.
Assume inductively that $R_n$ is a truncation closed subring of 
$\smallk((t^\Gamma))$ . 
For $f\in R_n$ and $\gamma\in \Gamma$ we have $\partial(f)|_\gamma = \partial(f|_\gamma)$, so all truncations of $\partial(f)$ lie in $R_{n+1}$, and thus $R_{n+1}$ is truncation closed by Proposition~\ref{D}(ii). 
Since the differential ring generated by $R$ is $\bigcup_{n} R_n$, and a union of truncation closed subsets of $\smallk((t^\Gamma))$ is truncation closed, we conclude that the differential subring of $\smallk((t^\Gamma))$ generated by $R$ is truncation closed. 
\end{proof}

\noindent
Note that Lemma~\ref{DifFieldGenTrunc} goes through with ``subfield'' instead of
``subring''.


\subsection*{Adjoining solutions to $y-ay'=f$}
We wish to preserve truncation closedness under
adjoining solutions to differential equations $y-ay'=f$, where 
$a,f \in \smallk((t^\Gamma))$, $a\prec 1$.
This differential equation is expressed more suggestively as
$(I-a\partial)(y)=f$, where $I$ is the identity operator on  
$\smallk((t^\Gamma))$ and $a\partial$ is considered as a (strongly additive)
operator on $\smallk((t^\Gamma))$ in the usual way. Note that $a\partial$ 
is small as defined earlier, hence
$I-a\partial$ is bijective, with inverse 
\[(I-a\partial)\inv\ =\ \sum_n (a\partial)^n\]
Thus the above differential equation has a unique solution 
$y=(I-a\partial)\inv(f)$ in $\smallk((t^\Gamma))$. This is why we now turn our attention to the operator $(I-a\partial)\inv$.

\medskip\noindent
For $n\geq 1$, $0\leq m\leq n$ we define $G^n_m(X)\in \mathbb{Z}\{X\}$ recursively as follows:
\begin{itemize}
\item $G^n_0 = 0$,
\item $G^n_n = X^n$,
\item $G^{n+1}_m = X(\partial(G^n_m) +G^{n}_{m-1})$ for $1\leq m\leq n$. 
\end{itemize}
This recursion easily gives
\[(a\partial)^n\ =\ \sum_{m=1}^n G^n_m(a)\partial^m \qquad(n\ge 1),\]
hence
\begin{equation}
(I-a\partial)\inv = I + \sum^{\infty}_{n=1} \sum_{m=1}^n G^n_m(a)\partial^m.
\end{equation}

\noindent Since $G^n_m(X)$ is homogeneous of degree $n$, we have
$\supp G^n_m(a)\in \smallk t^{n\alpha}$ for $a\in \smallk t^\alpha$.

\begin{lemma}\label{OpLemma1}
Suppose $R$ is a truncation closed differential subring of
$\smallk((t^\Gamma))$, $a\in R\cap\smallk t^\Gamma$, $a\prec 1$, $f\in R$, and  $(I-a\partial)\inv(g)\in R$ 
for all $g\truncof f$. Then all proper truncations of $(I-a\partial)\inv (f)$ lie in $R$.
\end{lemma}

\begin{proof} We have 
$(I-a\partial)\inv(f) = f+  \sum^{\infty}_{n=1} \sum_{m=1}^n G^n_m(a)\partial^m(f)$.
Also $a=a_{\alpha}t^\alpha$ with $\alpha>0$, hence $\supp  \left(G^n_m(a)\partial^m(f)\right)\subseteq n\alpha+\supp f$ for 
$1\le m\le n$. Consider a proper truncation $(I-a\partial)\inv(f)|_{\gamma}$
of $(I-a\partial)\inv(f)$; we have to show that this truncation lies in $R$.
The truncation being proper gives $N\in \mathbb{N}^{\ge 1}$ and $\beta\in \supp f$ with 
$\gamma\le N\alpha +\beta$. 
Let $f_1:= f|_\beta$ and $f_2:=f-f_1$. Then $f_1, f_2\in R$ and
$$
(I-a\partial)\inv(f)\ =\ (I-a\partial)\inv(f_1) + f_2 + \sum^{\infty}_{n=1} \sum_{m=1}^n G^n_m(a)\partial^m(f_2).$$
Using $\supp f_2\ge \beta$ and truncating at $\gamma$ gives 
\[(I-a\partial)\inv(f)|_{\gamma}\ =\ (I-a\partial)\inv(f_1)|_\gamma + f_2|_\gamma + \left.\left(\sum^{N-1}_{n=1} \sum_{m=1}^n G^n_m(a)\partial^m(f_2)\right)\right|_\gamma,\]
which lies in $R$, since $f_1\truncof f$ and thus $(I-a\partial)\inv(f_1)\in R$. 
\end{proof}

\noindent
The key inductive step is provided by the next lemma. 

\begin{lemma}\label{OperatorLemma} Let $R$ be a truncation closed differential subring of $\smallk((t^\Gamma))$. Let $a,f\in R$ be such that $a\prec 1$ and for all $b,g\in R$,
\begin{itemize}
\item $g\truncof f \Rightarrow (I-a\partial)\inv (g)\in R$,
\item $b\truncof a \Rightarrow (I-b\partial)\inv(g)\in R$.
\end{itemize}
Then all proper truncations of $(I-a\partial)^{-1}(f)$ lie in $R$.
\end{lemma}

\begin{proof} Assume $(I-a\partial)^{-1}(f)|_\gamma\truncof (I-a\partial)^{-1}(f)$. Then we have
$\alpha\in \supp a$,  $\beta\in \supp f$, and $N\in \mathbb{N}^{\ge 1}$ such that
 $\gamma\leq N\alpha + \beta$. 
Put $a_1:=a|_{\alpha}$, $a_2:= a-a_1$, $f_1:=f|_\beta$, and $f_2:=f-f_1$. 
Set $P=a_1\partial$ and $Q=a_2\partial$, so $a\partial=P+Q$.  
Then
\begin{align*}
(I-a\partial)\inv(f)\ &=\ (I-a\partial)\inv(f_1) + (I-a\partial)\inv(f_2)\\
&=\ (I-a\partial)\inv(f_1) + \sum_{n=0}^\infty (P+Q)^n(f_2).
\end{align*}
Note that 
\[\sum_{n=0}^{\infty}(P+Q)^n\ =\ \sum_{n_{0}} P^{n_{0}} + \sum_{n_{0},n_{1}} P^{n_{0}}QP^{n_{1}} +\hspace{-0.3cm} \sum_{n_{0},n_{1},n_{2}} \hspace{-0.3cm}P^{n_{0}}QP^{n_{1}}QP^{n_{2}} + \cdots.\]
We have $Q(I-P)\inv=\sum_n QP^n$; raising both sides to the $m$th power gives 
\begin{align*} \big(Q(I-P)\inv\big)^m\ &=\ \sum_{n_1,\dots, n_m}QP^{n_1}QP^{n_2}\cdots QP^{n_m}, \text{ so}\\ \sum_{n=0}^{\infty}(P+Q)^n\ &=\ \sum_{m=0}^\infty (I-P)\inv \big(Q(I-P)\inv\big)^m,\ \text{  hence}\\
 (I-a\partial)\inv(f)\ &=\ (I-a\partial)\inv(f_1) + \sum_{m=0}^\infty (I-P)\inv \big(Q(I-P)\inv\big)^m(f_2). 
\end{align*}
Truncating at $\gamma$ yields
\[(I-a\partial)\inv(f)|_{\gamma}\ =\ (I-a\partial)\inv(f_1)|_{\gamma} + \left.\left(\sum_{m=0}^{N-1}(I-P)\inv (Q(I-P)\inv)^m(f_2)\right)\right|_\gamma.\] 
Since $f_1\truncof f$ the first summand of the right hand side lies in $R$. Since $a_1\truncof a$, we have $(I-P)\inv(h)\in R$ for all $h\in R$, and thus, using $a_2, f_2\in R$, 
\[\sum_{m=0}^{N-1}(I-P)\inv (Q(I-P)\inv)^m (f_2)\in R.\]
Therefore $(I-a\partial)^{-1}(f)|_\gamma\in R$. 
\end{proof}

\begin{theorem}\label{thmA} Let $E$ be a truncation closed differential subfield
of $\smallk((t^\Gamma))$. Let $\widehat{E}$ be the smallest differential subfield of 
$\smallk((t^\Gamma))$ that contains $E$ and is closed under $(I-a\partial)\inv$ for all $a\in \widehat{E}^{\prec 1}$. Then $\widehat{E}$ is truncation closed.
\end{theorem}

\begin{proof}
Let $F$ be a maximal truncation closed differential subfield of $\widehat{E}$ 
containing $E$. (Such $F$ exists by Zorn's Lemma.) It suffices to show that 
$F= \widehat{E}$. Assume $F\ne \widehat{E}$. Then there exist  $a\in F^{\prec 1}$ and $f\in F$ with 
$(I-a\partial)\inv(f) \notin F$. Take such $a$ and $f$ for which
$(\alpha, \beta)$ is minimal for lexicographically ordered pairs of ordinals,
where $\alpha$ is the order type of the support of $a$ and $\beta$ that of
 $f$. By minimality of $(\alpha,\beta)$ we can apply Lemma \ref{OperatorLemma} to $F,a$, and $f$ to get that $F\big((I-a\partial)\inv(f)\big)$ is 
truncation closed, 
using also parts (i) and (ii) of Proposition~\ref{D}. Since $F\big((I-a\partial)\inv(f)\big)$ is a 
differential subfield of $\widehat{E}$, this contradicts the maximality of $F$. 
\end{proof}

\subsection*{Complementing the previous results}
Suppose $\smallk_1$ is a differential 
subfield of $\smallk$ and $E$ is a
subfield of 
$\smallk((t^\Gamma))$ such that $a\in \smallk_1$ and
$c(\gamma)\in \smallk_1$ whenever $at^\gamma\in E$, $a\in \smallk^\times$, $\gamma\in \Gamma$. Let $\Gamma_1$ be the subgroup of $\Gamma$ generated by the
$\gamma\in \Gamma$ with $at^\gamma\in E$ for some $a\in \smallk^\times$. 
In connection with Lemma~\ref{DifFieldGenTrunc} we note:

\medskip\noindent
{\em The differential subfield of $\smallk((t^\Gamma))$ generated by $E$ is 
contained in  $\smallk_1((t^{\Gamma_1}))$.}

\medskip\noindent
This is because $E\subseteq \smallk_1((t^{\Gamma_1}))$ and $\smallk_1((t^{\Gamma_1}))$ is a differential subfield of 
$\smallk((t^\Gamma))$. Likewise we can complement Theorem~\ref{thmA}:

\medskip\noindent
{\em If
$\partial E\subseteq E$ and $\widehat{E}$ is the smallest 
differential subfield of $\smallk((t^\Gamma))$ that contains $E$ and is closed under
$(I-a\partial)\inv$ for all $a\in \widehat{E}^{\prec 1}$,   
then $\widehat{E}\subseteq \smallk_1((t^{\Gamma_1}))$.}

\medskip\noindent
This is because $\smallk_1((t^{\Gamma_1}))$ is closed under
$(I-a\partial)\inv$ for all $a\in\smallk_1((t^{\Gamma_1}))^{\prec 1}$. 

\subsection*{Exponentiation} Our aim is to apply the material above
to the differential field $\T_{\exp}$ of purely exponential transseries. 
The construction of $\T_{\exp}$ involves an
iterated formation of Hahn fields, where at each step we apply
 the following general procedure (copied from \cite{DMM}). 

\medskip\noindent
Define a {\bf pre-exponential ordered field}
to be a tuple~$(E, A, B, \exp)$ such that
\begin{enumerate}
\item $E$ is an ordered field;
\item $A$ and $B$ are additive subgroups of $E$ with $E=A\oplus B$ and $B$
convex in $E$;
\item $\exp\colon B\to E^{\times}$ is a strictly increasing group morphism (so $\exp(B)\subseteq E^>$).
\end{enumerate}
Let $(E,A,B,\exp)$ be a pre-exponential ordered field. We view $A$ as the part of $E$ where exponentiation is not yet defined, and accordingly we introduce
a ``bigger'' pre-exponential ordered field $(E^*,A^*,B^*,\exp^*)$ as follows:
Take a {\em multiplicative\/} copy $\exp^*(A)$ of the ordered additive
group $A$ with order-preser\-ving isomorphism 
$\exp^*\colon A\to\exp^*(A)$,
and put $E^*\ :=\ E[[\exp^*(A)]]$. Viewing $E^*$ as an ordered Hahn field over 
the ordered coefficient field $E$, we set
$$ A^*\ :=\ E[[\exp^*(A)^{\succ 1}]],\qquad B^*\ :=\ (E^*)^{\preceq 1}\ =\ E\oplus 
(E^*)^{\prec 1}\ =\ A \oplus B \oplus (E^*)^{\prec 1}.$$ 
Note that $\exp^*(A)^{\succ 1}=\exp^*(A^{>})$. Next we extend
$\exp^*$ to $\exp^*\colon B^*\to (E^*)^{\times}$ by
$$\exp^*(a+b+\varepsilon)\ :=\ \exp^*(a)\cdot \exp(b)\cdot \sum_{n=0}^\infty
\frac{\varepsilon^n}{n!} \qquad(a\in A,\ b\in B,\ \varepsilon\in (E^*)^{\prec 1}).$$
Then $E\subseteq B^*=\operatorname{domain}(\exp^*)$, and $\exp^*$ extends $\exp$. Note that $E<(A^*)^>$ (but $\exp^*(E)$ is cofinal in $E^*$ if $A\ne \{0\}$). In particular, for $a\in A^{>}$, we have 
$$\exp^*(a)\in \exp^*(A^{>})\ \subseteq\ (A^*)^>,\ \text{  so $\exp^*(a)\ >\ E$.}$$

\medskip\noindent
Assume also that a derivation $\partial$ on the field $E$ is given that respects
exponentiation, that is, $\partial(\exp(b))=\partial(b)\exp(b)$ for all 
$b\in B$. Then we extend $\partial$ uniquely to a strongly $E$-linear 
derivation $\partial$ on 
the field $E^*$ by requiring that $\partial(\exp^*(a))=\partial(a)\exp^*(a)$
for all $a\in A$. This falls under the general construction at the beginning of this section with $\smallk=E$ and $\fM= \exp^*(A)$ by taking the additive function 
$c: \exp^*(A) \to E$ to be given by $c(\exp^*(a))=\partial(a)$.  
It is also easy to check that this extended derivation on $E^*$ again respect
exponentiation in the sense that  
$\partial(\exp(b))=\partial(b)\exp(b)$ for all 
$b\in B^*$.


\section{Directed Unions of Hahn fields}
\noindent
Let $\smallk$ be a field and $\fM$ a (multiplicative) ordered abelian 
group with distinguished subgroups $\fM_n$, $n=0,1,2,\dots$, such that, with
$\fM^{(n)}:=\fM_0 \cdots \fM_n\subseteq \fM$, we have:\begin{enumerate}
\item $\fM=\bigcup_n \fM^{(n)}$;
\item $\fM_m \prec \fm_n$ for all $m<n$ and $\fm_n\in \fM^{\succ 1}_n$.
\end{enumerate} 
Thus $\fM^{(n)}$ is a convex subgroup of $\fM$. Considering subgroups of $\fM$ 
as 
ordered subgroups, we have the anti-lexicographically ordered internal 
direct product
$$\fM^{(n)}\ =\ \fM_0 \times \cdots \times \fM_n\ \subseteq\ \fM.$$
Setting $\smallk_n:= \smallk[[\fM^{(n)}]]$ we have 
$\smallk_0=\smallk[[\fM_0]]$ and field extensions
forming a chain
$$\smallk\subseteq \smallk_0\subseteq \smallk_1\subseteq \cdots\subseteq \smallk_n\subseteq \smallk_{n+1}\subseteq \cdots$$
of Hahn fields over $\smallk$. We also identify $\smallk_{n}$ with the
Hahn field $\smallk_{n-1}[[\fM_n]]$ over $\smallk_{n-1}$ in the usual way
(where $\smallk_{-1}:=\smallk$ by convention). We set 
$$\smallk_{*}\ :=\ \bigcup_n \smallk_n, \quad\text{ a truncation closed subfield of the Hahn field }\smallk[[\fM]].$$
A set $S\subseteq \smallk_{*}$ is said to be truncation closed
if it is truncation closed as a subset of the Hahn field $\smallk[[\fM]]$.
Take an order reversing group isomorphism $v: \fM \to \Gamma$ onto an
additive ordered abelian group $\Gamma$. Then $v$ extends to the valuation
$v:\smallk[[\fM]]^\times \to \Gamma$ given by $v(f)=v(\max \supp f)$. Note that for 
$a\in \smallk_n= \smallk_{n-1}[[\fM_n]]$ we have: 
$$a\prec_{\fM_n}1\ \Longleftrightarrow\ va>v(\fM^{(n-1)}).$$

\medskip\noindent
We now assume that $\smallk$ is even a differential field, and that for
every $n$ there is given an additive map $c_n: \fM_{n} \to \smallk_{n-1}$.
Then we make $\smallk_n$ into a differential field by recursion on $n$:
$\smallk_n=\smallk_{n-1}[[\fM_n]]$ has the derivation given by the 
derivation of $\smallk_{n-1}$
and the additive map $c_n: \fM_{n} \to \smallk_{n-1}$. Thus $\smallk_{n}$ is a differential field extension of $\smallk_{n-1}$. 
We make $\smallk_{*}$ into a differential field so that it contains every
$\smallk_n$ as a differential subfield. It follows easily that
$$(\fm_0\cdots \fm_n)^\dagger\ =\  c_0(\fm_0) +\cdots + c_n(\fm_n), \qquad (\fm_0\in \fM_0, \dots, \fm_n\in \fM_n).$$ 
This suggest we define the additive function $c: \fM \to \smallk_*$ by 
$$c(\fm_0\cdots \fm_n)\ :=\ c_0(\fm_0) +\cdots + c_n(\fm_n) \qquad (\fm_0\in \fM_0, \dots, \fm_n\in \fM_n),$$
so $c$ extends each $c_n$, and $\fm^\dagger=c(\fm)$ for all $\fm\in \fM$.


\bigskip\noindent
We will say that this derivation is \textbf{transerial} if $c_n(\fM_n^{\neq 1})\succ_{\fM_{n-1}}1$ for $n\geq 1$.
One natural question we could ask at this point is the following:
Is the differential field generated by a truncation closed set inside $\smallk_*$ truncation closed? Unfortunately the answer is no.

\medskip\noindent
\textbf{Example:} Let $\alpha_0 > \alpha_ 1 > \cdots >\beta_0 > \beta_1 > \cdots 0$ be a decreasing sequence of real numbers such that $g:=\sum_n \alpha_nx^{\alpha_n}$ and $h:=\sum_n \beta_n x^{\beta_n}$ are differentially transcendental over $\R(x)$. Then the field $F=\R(x,f)$ with $f = \exp(\sum_n x^{\alpha_n} + \sum_n x^{\beta_n})$ is truncation closed, but $K=\R\langle x, f\rangle$, the differential ring generated by $F$, is not truncation closed. To see this, note that $g+h = f^{\dagger} \in K$, and if $K$ is truncation closed, then both $g$ and $h$ would be in $K$ making the differential transcendence degree of $F/\R(X)$ greater than $2$, a contradiction.


\medskip\noindent
Below we let $\fm_i$ lie in $\fM_i$ for $i\in \mathbb{N}$. For $f\in \smallk[[\fM]]$.
For $S$ a subset of $\smallk_{*}$ we set $\supp(S):= \bigcup_{f\in S}\supp(f)$. 
We say that $S$ is $\il$-closed if for all $\fm_0\cdots \fm_n \in \supp(S)$ we have $\int c_i(\fm_i), {\fm_i},\int c_1(\fm_1) + \cdots +\int c_n(\fm_n), \fm_1\cdots \fm_n \in S$ for $i \in \{1,\ldots,n\}$. We say that $S$ is $\til$-closed if $S$ is both truncation closed and $\il$-closed.

Let $\smallk S$ be the $\smallk$-linear subspace of $\smallk_{*}$ generated by $S$. Note that $\supp\smallk S = \supp(S)$ and if $S$ is $\til$-closed, then so is $\smallk S$ and $\supp(S) = \supp \smallk S \subseteq \smallk S$. Let $M(S):=(\supp{S})^*$ be the submonoid of $\fM$ generated by $\supp{S}$, and $G(S)$ the subgroup of $\fM$ generated by $\supp{S}$. Thus $\supp{\smallk[S]}\subseteq M(S)$.
The following hold. 

\begin{lemma}\label{tL0}
$\supp\smallk(S)\subseteq \Gamma(S)$
\end{lemma}
\begin{proof}
Let $g=f/h\in \smallk(S)$ with $f,h\in \smallk[S]$ since 
\[\supp(g)\subseteq \supp(f)\supp(1/h),\] 
it suffices to show that $\supp(1/h)\subseteq G(S)$ let $r\in \smallk, h\asymp \monm\in M(S)$ and $\epsilon\prec 1$ be such that $h=c\monm(1-\epsilon)$. Note that $\supp(\epsilon) \subseteq \monm\inv M(S)$. We have $1/h = (c\fm)\inv \sum_n \epsilon^n$, so $\supp(1/h)\subseteq G(S)$
\end{proof}

\begin{lemma}\label{tL1} If $S$ is $\til$-closed, then so is $\smallk S\cup M(S)$.
\end{lemma}
\begin{proof}
For $\monm=\fm_0\cdots\fm_n\in \supp(\smallk S)$ it is clear that $\int c_i(\fm_i), {\fm_i},\int c_1(\fm_1)+\cdots + c_n(\fm_n), \fm_1\cdots\fm_n\in S\subseteq \smallk S\cup M(S)$ for $i\in \{1,\ldots,n\}$. 
Let $\monm =\fm_0\cdots\fm_n\in M(S)$. Then there are 
$(\fm_{i,j})_{i,j}\in \fM$ such that $\sum_j \fm_{i,j} = \fm_i$ and ${\fm_{0,j} \cdots \fm_{n,j}}\in \supp(S)$ for $i=0,\ldots, n$, $j\in \{1, \ldots,m\}$. Hence $\int c_i(\fm_{i,j}) \in S$ and ${\fm_{1,j}\cdots \fm_{n,j}}\in S$ for $i=1,\ldots, n$, $j\in \{1, \ldots,m\}$, so $\int c_i(\fm_i),\int c_1(\fm_1) + \cdots \int c_n(\fm_n)\in \smallk S$ and ${\fm_i},{\fm_1\cdots \fm_n}\in M(S)$ for $i\in \{1,\ldots,n\}$. It remains o note that $\smallk S$ is truncation closed and so is $M(S)$  
\end{proof}

\begin{corollary}\label{tL2} If $S$ is $\til$-closed, then so is $\smallk[S]$.
\end{corollary}
\begin{proof} Assume $S$ is $\til$-closed. Then $\supp{S}\subseteq \smallk S\subseteq \smallk[S]$, so we have $\supp{\smallk[S]}=M(S)$. It remains to note that $\smallk[S]$ is truncation closed. 
\end{proof}
\noindent
Note that to check that a subfield $F$ of $\smallk_*$ is $\il$-closed it suffices to check that for all $n>0$ and $\fm_0\cdots\fm_n\in \supp(F)$ we have $\int c_n(\fm_n), {\fm_n}\in F$

\begin{corollary}\label{tL3} If $S$ is $\til$-closed, then $\smallk(S)$ is $\til$-closed and $G(S)\subseteq \smallk(S)$.
\end{corollary}
\begin{proof} Assume $S$ is $\til$-closed. Then $\smallk(S)$ is truncation closed, and so it follows that $G(S)=\supp{\smallk(S)}\subseteq \smallk(S)$.
If ${\fm_0 \cdots \fm_n}\in G(S)$, then ${\fm_n} \in G(S)$ and $\int c_n(\fm_n) \in \smallk S$.
\end{proof}

\begin{lemma}\label{tL5} Suppose $S$ is such that for each $\fm_0 \fm\in \supp(S)$ we have $\fm\in S$ and $\int c(\fm)\in S$. 
If $f\in \smallk[S]$, then every truncation of $f'$ lies in
$\sum\smallk[S]\smallk[S]'$, the $\smallk[S]$-submodule of $\smallk_*$ generated by $\smallk[S]'$.
\end{lemma} 
\begin{proof} Let $f\in \smallk_n\cap \smallk[S]$ and let $g$ be a truncation of $f'$; we prove by induction on $n$ that $g\in \sum\smallk[S]\smallk[S]'$. 
For $n=0$ we have
$f=\sum_{\fm\in \fM_0} f_\fm \fm$, and so $f'=\sum_{\fm \in \fM_0} (f_\fm' +c_0(\fm)f_\fm) \fm$. 
Then either $g=f'$ or $g=\sum_{\fm\in \fM^{\succ\fm_0}} (f_\fm' +c_0(\fm)f_\fm) \fm=(f_{|\fm_0})'$ for some $\fm_0\in \fM_0$, and as $f_{|\fm_0}\in \smallk[S]$ for all $\fm_0 \in \fM_0$, we get $g\in \smallk[S]'\subseteq \sum \smallk[S]\smallk[S]'$. 

Next, let $n>0$, and consider first the case $f=f_\fm \fm$ with $f_\fm\in \smallk_0$ and $\fm \in \fM_1\cdots\fM_n$. 
Then $f'=(f_\fm'+f_\fm c(\fm))\fm$, so $g=g_1+g_2$ where $g_1$ is a truncation of $f_\fm'\fm$ and
$g_2$ is a truncation of $f_\fm c(\fm)\fm$. Hence $g_1=((f_\fm)_{|\fm_0})'\fm$
with $\fm_0\in \fM_0$. Now $(f_\fm)_{|\fm_0}\fm$ is a truncation of
$f_\fm \fm=f\in \smallk[S]$, so $(f_\fm)_{|\fm_0}\fm\in \smallk[S]$. 
Therefore,
\begin{align*} \left((f_\fm)_{|\fm_0}\fm\right)'\ &=\ g_1 + (f_\fm)_{|\fm_0}c(\fm)\fm\in \smallk[S]',\ \text{ with} \\
(f_\fm)_{|\fm_0}c(\fm)\fm\ &=\ (f_\fm)_{|\fm_0}\fm c(\fm)\in \smallk[S]S'\quad\text{(using ${\int} c(\fm)\in S$)},
\end{align*}
so $g_1\in \sum \smallk[S]\smallk[S]'$. From $\int c(\fm)\in S\cap E_{n-1}$, we get by induction that every truncation of $c(\fm)$ lies in $\sum \smallk[S]\smallk[S]'$, hence every truncation of
$f_\fm c(\fm)\fm=f_\fm \fm c(\fm)$ lies in $\sum \big(\smallk[S]\sum \smallk[S]\smallk[S]')= \sum \smallk[S]\smallk[S]'$, 
in particular, $g_2\in \sum\smallk[S]\smallk[S]'$. Therefore, $g=g_1+g_2\in \sum\smallk[S]\smallk[S]'$.

 In general, $f=\sum_{\fm\in \fM_1\cdots \fM_n} f_\fm \fm$ with all $f_\fm\in \smallk_0$. 
 Then $$f'=\sum_\fm (f_\fm'+ f_\fm c(\fm))\fm.$$ 
 Then $g=f'$, or for some $\fn$ we have
 $$g=\left(\sum_{\fm \in \fM^{\succ\fn}} (f_\fm'+f_\fm c(\fm))\fm\right)+ h,$$ where $h$ is a truncation of $(f_\fn \fn)'$. It remains to note that 
 $$\sum_{\fm \in \fM^{\succ\fn}} (f_\fm'+f_\fm c(\fm))\fm\ =\ \left(\sum_{\fm\in \fM^{\succ\fn}}f_\fm \fm \right)',$$
that $\sum_{\fm\in\fM^{\succ\fn}} f_\fm \fm$ is a truncation of $f$ (and thus in $\smallk[S]$), and that $f_\fn \fn$ is a difference of such truncations, and thus in $\smallk[S]$ as well. 
\end{proof}


\begin{lemma}\label{tL6} Suppose $S$ is $\il$-closed. Then $\supp(S')\subseteq \supp(S)^*$.
\end{lemma}
\begin{proof}
Let $f\in S\cap \smallk_n$, and $g\in \supp(f')$. We show, by induction on $n$, that $g\in \supp(S)^*$. If $n=0$ then the $g\in \supp(f)\subseteq \supp(S)$. 
Let $n>0$, so $f= \sum_{\fm \in \fM_0  \cdots \fM_n} f_\fm \fm$ with all $f_\fm\in \smallk$. Then $f'=\sum_\fm(f_\fm'+f_\fm c(\fm))\fm$, so we get $\fm$ with 
$$g\in t^\fm\supp(f_\fm'+f_\fm c(\fm))\subseteq \{\fm\}\cup \fm\supp(c(\fm)). $$
Since $S$ is $\il$-closed $\fm,\int c(\fm)\in S$. So by induction $\supp(c(\fm))\subseteq \supp(S)^*$. Hence $g\in \supp(S)^*$.
\end{proof}

\begin{corollary}\label{tL7} If $S$ is $\til$-closed, then so is $\smallk\{S\}$.
\end{corollary}
\begin{proof} Assume $S$ is $\til$-closed and set 
$$S_1\ :=\ \smallk[S]\cup \{g:\ g \text{ is a truncation of $f'$ for some $f\in S$}\}.$$
Let $f\in S$ be given. 
Then $\supp(f')\subseteq \supp(S)^*$ by Lemma~\ref{tL6}, and so 
if ${\fm_0\cdots \fm_n}\in \supp(f')$, 
then ${\fm_0\cdots\fm_n},{\fm_i}\in \smallk[S]$ and 
hence $\int c(\fm_0) +\cdots + \int c(\fm_n),\int c(\fm_i)\in \smallk[S]$ for $i\in \{1,\ldots,n\}$ by Corollary~\ref{tL2}. 
Thus $S_1$ is $\til$-closed by another use of Corollary~\ref{tL2}. Moreover, by Lemma~\ref{tL5} we have $S\cup S'\subseteq S_1\subseteq \sum \smallk[S]\smallk[S]'\subseteq \smallk\{S\}$. This leads to an increasing sequence
$$S=S_0\subseteq S_1 \subseteq S_2 \subseteq \cdots$$
of $\til$-closed subsets of $\smallk\{S\}$, where for each $n$,
$$S_{n+1}:= \smallk[S_n]\cup \{g:\ g \text{ is a truncation of $f'$ for some $f\in S_n$}\},$$
and $S^{(n)}\subseteq S_n$. Thus $S_{\infty}:= \bigcup_n S_n$
is $\til$-closed, and $\smallk\{S\}=S_{\infty}$. 
\end{proof}

\noindent
In view of Corollary~\ref{tL3}, this yields:

\begin{corollary}\label{tL8} If $S$ is $\til$-closed, then so is 
$\smallk\langle S\rangle$.
\end{corollary}

\subsection{A variant} In the results above we used $\smallk_{*}$ with the derivation $\partial$ as the ambient differential field. Let us fix an monomial $\mathfrak{e}$ and consider
instead $\smallk_{*}$ equipped with the derivation $\derdelta=\mathfrak{e}\partial$, and let $\smallk\{S;\mathfrak{e}\}$ be the differential
subring of $(\smallk_*, \derdelta)$ generated by $S$ over $\smallk$, and
likewise, let $\smallk\<S;\mathfrak{e}\>$ be the differential
subfield of $(\smallk_*, \derdelta)$ generated by $S$ over $\smallk$.

Lemmas~\ref{tL5} and ~\ref{tL6} then extends as follows:

\begin{lemma}\label{tL9} If $S$ is $\til$-closed and 
$f\in \smallk[S]$, then every truncation of $\derdelta(f)$ lies in
$\sum\smallk[S]\derdelta(\smallk[S])$, the $\smallk[S]$-submodule of $\smallk_*$ generated by $\derdelta(\smallk[S])$. 
\end{lemma}
\begin{proof} Immediate from Lemma~\ref{tL5} and the fact that
for $f\in \smallk_*$ any truncation of $\derdelta(f)=\mathfrak{e}f'$ 
equals $\mathfrak{e}g$ for some truncation $g$ of $f'$.
\end{proof}

\begin{lemma}\label{tL10} If $S$ is $\il$-closed and 
$\mathfrak{e}\in \supp(S)^*$, then 
\[\supp\derdelta(S)\subseteq \supp(S)^*\].
\end{lemma}

\begin{corollary}\label{tL11} Suppose $S$ is $\til$-closed and $\mathfrak{e}\in \supp(S)^*$. Then
$\smallk\{S;\mathfrak{e}\}$ and $\smallk\<S;\mathfrak{e}\>$ are $\til$-closed.
\end{corollary}
\begin{proof} Analogous to the proof of Corollary~\ref{tL7}, with $S_1$ replaced by 
$$ \smallk[S]\cup \{g:\ g \text{ is a truncation of $\derdelta(f)$ for some $f\in S$}\},$$
and Lemmas~\ref{tL6} and ~\ref{tL7} replaced by Lemmas~\ref{tL9} and ~\ref{tL10}. \end{proof}

\subsection{Operators of the form $(I-a\partial)\inv$ on $\smallk_*$}

\begin{lemma} Let $F$ be a truncation closed
differential subfield of $\smallk_{*}$ that contains
$\smallk$ and $\fM$. Let $F_{\infty}$ be the smallest differential subfield
of $\smallk_{*}$ that contains $F$ such that for every $n$,
$F_n:=F_{\infty}\cap \smallk_n$ is closed under
$(I-a\partial)\inv$ for all $a\in F_n$ with $a\prec\fM^{(n-1)}$, where
by convention $F_{-1}=\smallk_{-1}$ and $\fM^{(-1)}:=\{0\}$. 
Then $F_{\infty}$ is truncation closed.
\end{lemma}
\begin{proof} First note that $c_n(\fM_n)\subseteq F\cap \smallk_{n-1}$: 
this is because for $\fm\in \fM_n$ we have $\fm\in F$, so
$(\fm)'=c_n(\fm)\fm\in F$, hence $c_n(\fm)\in F\cap \smallk_{n-1}$.
 
We define differential subfields
$K_n\subseteq \smallk_n$ by recursion on $n$ as follows:
$K_n$ is the smallest differential subfield of $\smallk_n$ that
contains $F\cap \smallk_n$ and $K_{n-1}$ and is closed under
$(I-a\partial)\inv$ for all $a\in K_n$ with $a\prec\fM^{(n-1)}$, where
by convention $K_{-1}=\smallk_{-1}$. We show by induction on $n$:
$K_{n-1}({\fM_n}) \subseteq K_n \subseteq K_{n-1}[[{\fM_n}]]$  
and $K_n$ is truncation closed. The first inclusion holds because
$K_{n-1}\subseteq K_n$ and ${\fM_n}\subseteq F\cap \smallk_n\subseteq K_n$.
Also $c_n(\fM_n)\subseteq K_{n-1}$, so  $K_{n-1}[[{\fM_n}]]$ is a 
differential subfield of $\smallk_n$ closed under 
$(I-a\partial)\inv$ for all $a\in K_{n-1}[[{\fM_n}]]$ with $a\prec\fM^{(n-1)}$.
Moreover, by Lemma~\ref{trtr1}, 
$$F\cap \smallk_n\subseteq (F\cap \smallk_{n-1})[[{\fM_n}]]\subseteq 
K_{n-1}[[{\fM_n}]],$$
and so $K_n\subseteq K_{n-1}[[{\fM_n}]]$. Now the subfield
$E_n$ of $K_{n-1}[[{\fM_n}]]$ generated by $K_{n-1}$ and $F\cap \smallk_n$
is a differential subfield of  $K_{n-1}[[{\fM_n}]]$ and is also
$K_{n-1}$-truncation closed. Applying Theorem~\ref{thmA} to $E_n$ in the role of $E$ we conclude that $K_n$ is $K_{n-1}$-truncation closed. In view of
$K_{n-1}\subseteq K_n$ and Lemma~\ref{trtr2} it follows that $K_n$ is 
truncation closed.
  It is also clear by induction on $n$ that $K_n\subseteq F_n$.
Hence $K_{\infty}:= \bigcap_n K_n$ is a differential subfield of $F_{\infty}$
that contains $F$. Moreover, $K_{\infty}\cap \smallk_n=K_n$ is closed under
$(I-a\partial)\inv$ for all $a\in K_n$ with $a\prec\fM^{(n-1)}$,
so $K_{\infty}=F_{\infty}$ by the minimality of $F_{\infty}$.
Since $K_{\infty}$ is truncation closed, so is $F_{\infty}$.
\end{proof}

\noindent
The assumption that ${\fM}\subseteq F$ is too strong, but
we are going to replace it by something more realistic. We say that a set
$S\subseteq \fM$ is {\em neat} if for all ${\fm_0\cdots \fm_n}\in S$,
with $\fm_i\in \fM_i$ for $i=0,\dots,n$, we have
${\fm_i}\in S$ for $i=0,\dots,n$. 

\begin{lemma}\label{opClosure} Let $F$ be a truncation closed
differential subfield of $\smallk_{*}$ that contains
$\smallk$ and such that $F\cap {\fM}$ is neat. 
Let $F_{\infty}$ be the smallest differential subfield
of $\smallk_{*}$ that contains $F$ such that for every $n$,
$F_n:=F_{\infty}\cap \smallk_n$ is closed under
$(I-a\partial)\inv$ for all $a\in F_n$ with $a\prec\fM^{(n-1)}$. 
Then $F_{\infty}$ is truncation closed and 
$F_{\infty}\cap {\fM}=F\cap {\fM}$.
\end{lemma}
\begin{proof} Set $\fM_{F,n}:= (\supp F) \cap \fM_n$, a subgroup of
$\fM_n$, and $\fM_F:=\fM\cap F$, a subgroup of $\fM$. Setting
$\fM_F^{(n)}:=\fM_{F,n} + \cdots + \fM_{F,0}$, conditions (1) and (2) above hold for $\fM_F$, $(\fM_{F,n})$, $(\fM_F^{(n)})$ in place of
$\fM$, $(\fM_n)$, $(\fM^{(n)})$. We have the subfield 
$$\smallk_{F,n}\ :=\ \smallk[[{\fM_F^{(n)}}]]$$
of $\smallk_n= \smallk[[{\fM^{(n)}}]]$, with 
$F\cap \smallk_n\subseteq \smallk_{F,n}$. Using that $F$ is a 
differential subfield of $\smallk_{*}$ we get 
$c_n(\Gamma_{F,n})\subseteq\smallk_{F,n-1}$, and so $\smallk_{F,n}$ is a
differential subfield of $\smallk_n$. The increasing chain
$$\smallk\ \subseteq\ \smallk_{F,0}\ \subseteq\ \smallk_{F,1}\ \subseteq\ \cdots\subseteq\ \smallk_{F,n}\ \subseteq\ \smallk_{F,n+1}\ \subseteq \cdots$$
yields a differential subfield $\smallk_{F,\infty}:=\bigcup \smallk_{F,n}$
of $\smallk_{*}$ with $F\subseteq \smallk_{F,\infty}$. 
In view of $t^{\Gamma_F}\subseteq F$ it remains to apply the previous lemma with
$\smallk_{F,\infty}$ instead of $\smallk_{*}$.  
\end{proof} 

\noindent
One can probably also get rid of the assumption $\smallk\subseteq F$,
and assume instead that $F$ is strongly truncation closed, with the role
of $\smallk$ taken over by $\smallk_F:= F\cap \smallk$.

\begin{lemma}Let $F$ be a truncation closed differential subfield of $\smallk_{*}$ that contains $\smallk$ and such that $F\cap {\fM}$ is neat. 
Let $F_{\infty}$ be the smallest differential subfield
of $\smallk_{*}$ that contains $F$ such that for every $n$,
$F_n:=F_{\infty}\cap \smallk_n$ is closed under
$(I-a\partial)\inv$ for all $a\in F_n$ with $a\prec\fM^{(n-1)}$, and such that for every $f\in F_\infty$ such that $1\notin\supp(f)$ we have $g\in F$ such that $\partial(g)=f$. 
Then $F_{\infty}$ is truncation closed and 
$F_{\infty}\cap {\fM}=F\cap {\fM}$.
\end{lemma}

\begin{proof} Let $K$ be a maximal truncation closed differential subfield of $F_\infty$ with respect to that contains $F$ and such that for every $n$, $K_n:=K\cap\smallk_n$ is closed under $(I-a\partial)\inv$ for all $a\in K_n$ with $a\prec\fM^{(n-1)}$. We will show that $K$ is $F_\infty$. 
Assume otherwise. Consider $S=\{f\in K : f\notin \partial(K),\ 1\notin \supp(f)\}$. 
Let $f$ be minimal in $S$ with respect to the order type of its support. 
Then there is $n$ such that $f\in \smallk_n$ so $f= \sum_{\fm_n\in \fM_n}f_{\fm}{\fm}$.
Thus $g= \sum_{\fm\in \fM_n}\left(I- \frac{-\partial}{c(\fm)}\right)\inv (f_{\fm}){\fm}$ is such that $\partial(g) =f$. {\sc Claim:} $K[g]$ is truncation closed.
To prove this it suffices to show that any truncation of $g$ lies in $K$. 
Let $\fn = \fn_0 + \cdots \fn_n \in \fM$ be such that $\fn_i \in \fM_i$ and $g|_\fn \truncof g$. 
Then 
\[g|_\fn = \left(\sum_{\fm\in \fM_n^{\succ\fn_n}}\left(I- \frac{-\partial}{c(\fm)}\right)\inv (f_{\fm}){\fm} \right)+ \left.\left(I- \frac{-\partial}{c(\fn_n)}\right)\inv (f_{\fn_n}){\fn_n}\right|_\fn\]
The first summand of the left hand side is in $K$ by the choice of $f$. 
The second summand is in $K$ since $f_{\fn_n} \in K$, $K$ is truncation closed and ${\fn_n}$ is in $F$. It follows that $K\langle g \rangle$ is truncation closed and $K\langle g \rangle \cap \fM$ is neat. Thus by Lemma \ref{opClosure} applied to $K\langle g \rangle$ in place of $F$ we get a contradiction with the maximality of $K$.  
\end{proof}

\section{Application to Transseries}
\label{transeries}
Let us consider the field of exponential transseries $\mathbb{T}_{\exp}$ as in \cite[Appendix A]{ADH} together with the derivative $\partial : = x\frac{d}{dx}$ we consider $\mathbb{T}_{\exp}$ as a differential field. In \cite{ADH} $G_n$ takes the role of ${\fM^{(n)}}$.

\subsection*{The construction of $\mathbb{T}_{\exp}$} Let $\smallk = \R$, so that $\R[[\mathfrak{G}]]$ is considered as an ordered field for any ordered multiplicative group $\mathfrak{G}$, by setting $0<f$ if $f_\fm >0$ for $\fm = \max(\supp(f))$. 
The construction of $\mathbb{T}_{\exp}$ is the same as the construction of $\smallk_*$ in the previous section, where $\fM_0 = x^\R = \{x^r: r\in \R\}$ and $\fM_{n+1} = \exp*(\smallk_{n-1}[[{\fM_n^{<0}}]])$, where $\smallk_{n-1}[[{\fM_n^{<0}}]]$ is taken as an ordered additive subgroup of $\smallk_n$. 
By initially taking the trivial derivation on $\R$, we construct the derivation on $\smallk_*$ by inductively defining the map $c_n$ for every $n$. 
For $n=0$, $c_0 = x^r\mapsto r: \fM_0 \rightarrow \R$, and  $c_{n+1}:= \exp*(f)\mapsto \partial_{c_{n}}(f):\fM_{n+1}\rightarrow \smallk_n$. Thus for $\fm = x^r \fm_1 \cdots \fm_n \in \Gamma$ we have $c(\gamma) = r+\partial(\log\fm_1 + \cdots + \log\fm_n)$ and thus the derivation on $\Texp$ is the strongly additive map extending $\fm \overset{\partial}{\mapsto} (r+\partial(\log\fm_1 + \cdots + \log\fm_n))\fm$. 


We recall some of the notation for certain subsets of $\Texp$ and its valued group and identify using our notation in $\smallk_*=\Texp$. $E_n$ is $\smallk_n$, $A_n$ corresponds to $\smallk_{n-1}[[{\fM_n^{\succ 1}}]]$, $G_n$ corresponds to ${\fM^{(n)}}$. 
The set of exponential transmonomials $G^E$ corresponds to $\bigcup_n {\fM^{(n)}}$. 
The convex subring $B_n$ corresponds to $\smallk_{n-1}[[{\fM^{\preccurlyeq 1}}]]$.

Thus we can translate the notions of $\il$-closed and $\til$-closed on subsets of $\mathbb{T}_{\exp}$. We set $\supp_{\exp}(f) :=\{E(a): x^rE(a)\in \supp (f)\}$ for $f\in \Texp$ and $\supp_{\exp}(S) :=\bigcup_{f\in S}\supp_{\exp}(f)$. A subset $S$ of $\Texp$ is $\il$-closed if for every exponential transmonomial $E(a_0 + \cdots a_n)\in \supp_{\exp}(S)$ with $a_i \in A_i$ for $i\in \{0,\ldots, n\}$ we have $E(a_i)$, $a_i$, $\sum_i a_i$, $E(\sum_i a_i)$ are all in $S$ for $i \in \{0,\ldots, n\}$. Thus as direct results from corollary \ref{tL8} we get that the differential field generated by $\R$ and any $\til$-closed subset of $\Texp$ is again $\til$-closed. 

\subsection*{The exponential map on $\mathbb{T}_{\exp}$} For $f=\ser{f}\in \mathbb{T}$ we set $f_{\prec} := f_{|1}$ the \textbf{purely infinite part} of $f$, $f_\asymp:=f_0\in \R$ the \textbf{constant term} and $f_\prec:= f-f_\succ -f_\asymp$ the \textbf{infinitesimal part} of $f$, and $f_{\preceq}:= f_\asymp + f_\prec$ the \textbf{bounded} part of $f$. The exponential map on $\mathbb{T}_{\exp}$ is obtained by 
\[\exp(f) = \exp(f_\asymp)\exp({f_\succ}) \left(\sum_{n} \frac{f_\prec^n}{n!}\right),\]
where $\exp(f_\asymp)$ is the image of $f_\asymp$ under the usual exponential map on $\R$ and $\exp(f_\succ)$ is the image of $t^{-f_\succ}$ under $\iota$.

\subsection*{Extension to $\mathbb{T}$} With $L_n = \R[[\ell_n^\R]]^E \cong \Texp$ as in \cite{ADH} we recall that $\mathbb{T}=\bigcup_n L_n$, and that for each $n$ we are given an automorphism $f\mapsto f\downer^n=f(\ell_n)$ of the exponential ordered field 
$\mathbb{T}$ that is the identity on $\R$, maps $\mathbb{T}^E$ onto $L_n$, maps $G^E$ onto $G^{LE}\cap L_n$, and preserves infinite sums. 
For $n=0$ it is the identity. The inverse of the automorphism $f\mapsto f\downer^n$ is $g\mapsto g\upper^n=g(e_n)$ with $e_0 = x$ and $e_{m+1} = \exp(e_m)$.
  Let $g\in L_n$; then $g=f(\ell_n)$ with $f=g\upper^n\in \mathbb{T}^E$, and this is a useful way
to think about $g$. Also $g\in L_{n+1}$, and thus likewise 
  $g=(f\upper)(\ell_{n+1})$. Before we extend the notion of $\il$-closedness to subsets of $\mathbb{T}$ we make an observation on $\il$-closed subsets of $\mathbb{T}_{\exp}$. 
  
  \begin{lemma}\label{tL12}
  Let $S$ be an $\il$-closed subset of $\Texp$ that contains a non-constant element. Then $S$ contains a non-constant element of $E_0$.
  \end{lemma}
  
  \begin{proof} 
  Each $f\notin E_0$ has a monomial of the form $x^r\exp(a)$ with $a \in A^{\neq}$, and $r\in \R$. Consider the set $M$ of such monomials as $f$ ranges in $S\setminus E_0$. If $M$ is empty then we are done. Otherwise take $n$ minimal such that $x^r\exp(a)$ is in $M$ and $a\in A_n\oplus \cdots \oplus A_0 \setminus A_{n-1}\oplus \cdots \oplus A_0$. Since $S$ is L-closed, $a\in S$. Note that $a$ is non-constant, so $a\in E_0\setminus \R$ or $a$ has a monomial $x^s\exp(b)$ with $b$ non-zero, which contradicts the minimality of $n$.
  \end{proof}
  We would like to extend the notion of $\il$-closedness to subsets of $\mathbb{T}$. For $f\in \mathbb{T}$ we define the \textbf{depth} of $f$ as the smallest $n$ such that $f\in L_n$. For $S\subseteq \T$ we define the depth of $S$ as the supremum of the depth of its elements in $|mathbb{N}\cup \{\infty\}$. If $S$ has depth $n$, then we say that $S$ is $\il$-closed if $S\upper^n\subseteq \Texp$ is $\il$-closed. If $S$ has depth $\infty$ then we say that $S$ is $\il$-closed if for all $n$ we have $S\upper^n\cap \Texp$ is $\il$-closed. As before, if $S$ is both $\il$-closed and truncation closed we say that it is $\til$-closed. 
  
\begin{lemma}\label{tL12.1}
Let $V$ be a $\til$-closed $\R$-vector subspace of $\Texp$. Then $V\upper + \R x$ is $\til$-closed.
\end{lemma}
\begin{proof}
We first note that $V\upper\cap \R x = \{0\}$ thus $V\upper + \R x$ is a direct sum. Since $V\upper$ is truncation closed, then $V\upper +\R x$ is truncation closed. Let $\exp(a_0,\ldots,a_n)\in \supp_{\exp}(V\upper +\R x)$. 
Then $\exp(a)$ is in the support of $V\upper$. 
One can easily check that $a = (rx+b\upper)$, for some $r\in \R$, $b=b_0 + \cdots +b_{n-1}$ with $b_i\in A_i$ for $i\in \{0,\ldots,n-1\}$ and with the element $x^r\exp(b)\in \supp (V)$. 
By $\il$-closedness of $V$ we get $\exp(b), b_i,b,\exp(b_i),b \in V$ for $i\in \{0,\ldots,n-1\}$, hence $b\upper, b_i\upper(=a_i),\exp(b_i)\upper = \exp(a_i) \in V\upper$ and thus $rx + b\upper \in V\upper + \R x$.
Since $V$ is truncation closed and a vector space over $\R$ we have that $x^r\exp(b)\in V$. 
Then $\exp(a) = x^r\exp(b)\upper \in V\upper\subseteq V\upper \R x$.
Thus concluding that $V + \R x$ is $\til$-closed.
\end{proof}

\begin{corollary}
Let $V$ be a truncation closed $\R$ vector subspace of $\mathbb{T}_{\exp}$. Then $V \oplus \R\ell_1 \oplus \cdots \oplus \R\ell_n$ is $\til_{m}$-closed.  
\end{corollary}

\noindent
Lemma \ref{tL12.1} together with \ref{D} give the following: 
\begin{corollary}
Suppose that $K$ is a $\til$-closed subfield of $\mathbb{T}_{\exp}$ containing $\R$. Then $K\upper (x)$ is $\til$-closed.
\end{corollary}

\noindent
This gives a nice result for certain special subsets of $\mathbb{T}$.
\begin{corollary}\label{logs1}
Let $K$ be a $\til$-closed subfield of $\mathbb{T}_{\exp}$ containing $\R$. Then the fields $K(\ell_1,\ldots,\ell_n)$ and $K(\{\ell_n\}_n) $ are $\til$-closed.
\end{corollary}
\noindent
This observations will be helpful further on in the analysis of integrals.

\noindent
Recall:
  $$(g')\upper^n\ =\ \frac{1}{e_1\cdots e_n}f'.$$

\begin{corollary}\label{tL13} If $S\subseteq L_n$ is $\til$-closed and $\ell_0^{-1},\dots, \ell_{n-1}^{-1} \in S$ or $S\supseteq \{\ell_n;n\in \mathbb{N}\}$. Then $\R\{S\}$ and
$\R\<S\>$ are $\til$-closed.
\end{corollary}
\begin{proof} By the formula above for derivatives, $g\mapsto g\upper^n: L_n \to \mathbb{T}^E$ is a differential field isomorphism where $L_n$ is equipped with the usual derivation and $\mathbb{T}^E$ is given the derivation $\mathfrak{e}_n\frac{d}{dx}$, where $\mathfrak{e}_n:=\frac{1}{e_1\cdots e_n}$.  It follows that 
$\R\{S\}\upper^n=\R\{S\upper^n;\mathfrak{e}_n\}$.
Since $\ell_0^{-1},\dots, \ell_{n-1}^{-1}\in S$ we have
$(\ell_0^{-1})\upper^n=e_n^{-1},\dots, (\ell_{n-1})\upper^n=e_1^{-1}$, all in $S\upper^n$, and so $\mathfrak{e}_n\in \supp(S\upper^n)^*$.
Hence $\R\{S\upper^n;\mathfrak{e}_n\}$ is $\til$-closed by Corollary~\ref{tL11}
It follows that $\R\{S\}$ is $\til$-closed.
\end{proof}

\subsection*{Henselization and Real Closure} It is easy to see that for a truncation closed additive subgroup $S$ of $\mathbb{T}$, we have $\supp(S)=\operatorname{lm}(S^{\neq})$. 
Thus using the fact that the henselization of a valued field is an immediate extension, together with proposition \ref{D} we get the following.

\begin{prop}\label{tL14} Suppose $K$ is a $\til$-closed subfield of $\T$ containing $\R$. Then the henselization of $K$ inside $\mathbb{T}$ is $\til$-closed.
\end{prop}

\noindent
We also have a similar result for the real closure.
\begin{lemma}\label{tL15} Let $K$ be a $\til$-closed subfield of $\T$ containing $\R$. Then the real closure of $K$ inside $\mathbb{T}$ is $\til$-closed.
\end{lemma}

\begin{proof}
Let us denote by $F$ the real closure of $K$ inside $\mathbb{T}$. By the previous proposition, we may assume that $K$ is henselian. 
Since the real closure is an algebraic extension, and algebraic extensions of henselian fields are truncation closed, we get that $F$ is truncation closed. 
Let $\exp(\sum_i a_i)\in \supp_{\exp}(F\upper^n)$. Then there is $m$ such that $\exp(m\sum_i a_i)\in \supp_{\exp}(K\upper^n)\subseteq K\upper^n$. 
Hence $ma_i,m\sum_ia_i\in K\upper^n$ gives $a_i,\sum_ia_i\in K\upper^n\subseteq F\upper^n$ for $i\in\{0,\ldots,n\}$. $E(a_i)$ is a root of $X^m-E(ma_i) \in K\upper^n[X]$ for $i\in\{0,\ldots,n\}$. Thus $E(a_i)\in F\upper^n$ for $i\in\{0,\ldots,n\}$.   
\end{proof}

\subsection*{Some transcendental extensions} Let $\mathcal{F}= (\mathcal{F}_n)$ be a family such that for each $n$: $\mathcal{F}_n\subseteq \R[[X_1,\ldots,X_n]]$, and for all $F\in \mathcal{F}_n$ we have $\partial F/\partial X_i \in \mathcal{F}_n$ for $i=1,\ldots,n$.
Let $K$ be a subfield of $\mathbb{T}$.
We define the $\mathcal{F}$-\textit{extension} of $K$, $K(\mathcal{F},\prec 1)$, to be
the smallest subfield of $\mathbb{T}$ that contains $K$ and the set 
$$\{F(f_1,\ldots,f_n): F\in \mathcal{F},\ f_i\in K,\ f_i\prec 1 \mbox{ for } i=1,\ldots,n \} .$$

\begin{lemma}\label{tL16}
Suppose $K$ is a truncation closed subfield of $\mathbb{T}$ containing $\R$. Then $K(\mathcal{F},\prec 1)$ is truncation closed. 
\end{lemma}

\begin{proof}
Let $E$ be the largest truncation closed subfield of $K(\mathcal{F},\prec 1)$.
We will proceed by contradiction, and assume that $K(\mathcal{F},\prec 1)\setminus E$ is nonempty.
Let $n$ be minimal such that there are $f_1,\ldots, f_n \in K^{\prec 1}$ and $F\in \mathcal{F}_n$ such that the element $F(f_1,\ldots,f_n)\notin E$.
Let $f_1,\ldots,f_n\in K^{\prec 1}$ be such that $(o(f_1),\ldots, o(f_n))$ is minimal in the lexicographic order and there is $F\in \mathcal{F}_n$ such that $F(f_1,\ldots,f_n)\notin E$.
Let $f = F(f_1,\ldots, f_n)$ with $F\in \mathcal{F}_n$ so $f \notin E$.
It suffices to show that all proper truncations of $f$ lie inside $E$.
If \[o(f_1)=\ldots = o(f_n) = 0\] then $F(f_1,\ldots, f_n)$ is a constant.
Let us assume that $f_i\neq 0$ for some $i\in \{1,\ldots,n\}$.
Let $\phi$ be a proper truncation of $f$.
Take $k\in \{1,\ldots,n\}$, $\monm\in \supp(f_k)$, and $N=N(\monm)\in \mathbb{N}$ such that $\monm^n < \supp(\phi)$ for all $n>N$.
Let $h,g \in K$ be such that $f_k=h+g$, $ h =f_k|_{\monm}$ and $\operatorname{lm}(g) = \monm$.
By Taylor expansion we have 
\[F(f_1,\ldots,f_n) = \sum_{n} \frac{\partial^n F}{\partial X_k^n}(f_1,\ldots,f_{k-1},h,f_{k+1},\ldots, f_n) \frac{g^n}{n!}.\]
Thus $\phi$ is a truncation of 
\[\sum_{n\leq N} \frac{\partial^n F}{\partial X_k^n}(f_1,\ldots,f_{k-1},h,f_{k+1},\ldots, f_n) \frac{g^n}{n!}.\]
Given that we took $(o(f_1),\ldots,o(f_n))$ minimal, and that $o(h) < o(f_k)$, we get $\phi \in E$.  
\end{proof}

\begin{lemma}\label{tL18}
Suppose $K$ is a $\til$-closed subfield of $L_n$ and $f\in K$ is such that for any proper truncation $g$ of $f$, we have $\exp(g) \in K$. Then $K(\exp(f))$ is $\til$-closed.
\end{lemma}

\begin{proof}
We first show that $K(\exp(f))$ is truncation closed.
For this, it is enough to show that all truncations of $\exp(f)$ lie inside $K(\exp(f))$.
We may assume that the infinitesimal part of $f$ is nonzero, otherwise $\exp(f) = \exp(f_\asymp)\exp(f_\succ)$ so the only truncations of $\exp(f)$ in this case are $0$ and $\exp(f)$.
Let $c$ be a truncation of 
$$\exp(f) = \exp(f_\asymp) \exp(f_\succ)\sum_n \dfrac{f_\prec^n}{n!}.$$
Then there is $\mathfrak{m}\in \supp(f_\prec)$ and $N=N(\mathfrak{m})\in \mathbb{N}$ such that $\exp(f_\succ)\mathfrak{m}^n\prec \supp(c)$ for all $n>N$.
Let $f_\prec = f_0 + f_1$ with $f_0 = f_\prec|_\mathfrak{m}$ and $\operatorname{Lm}(f_1)= \mathfrak{m}$.
Then  
\[\exp(f) = \exp(f_\asymp) \exp(f_\succ) \sum_n \frac{f_0^n}{n!}\sum_n\frac{f_1^n}{n!} = \exp(f_\asymp) \sum_n \frac{f_0^n}{n!}\sum_n\frac{\exp(f_\succ) f_1^n}{n!}\]
Thus $c$ is a truncation of $\exp(f|_\monm)\sum_{n=0}^N \frac{f_1^n}{n!}$. 
Now let $\exp(a)\in \supp_{\exp}(K(\exp(f))\upper^n)$ Since $\upper^n$ is an exponential field isomorphism it suffices to consider the case $K\subseteq \Texp$ to show that $K(\exp(f))$ is $\il$-closed. Moreover it is enough to show that if $\exp(a_0+\ldots+a_n)\in \supp_{\exp}(\exp(f))$ then $\exp(a_i),a_i\in K(\exp(f))$ for $i\in\{1,\ldots,n\}$. 
We consider two cases. 
First assume that $f_{\prec} = 0$. In this case $a_0+\cdots+a_n=f$ so both $\exp(a_n+\cdots +a_k),a_n+\cdots +a_k$ are in $K$ for $0<k\leq n$ by the assumptions in the Lemma, and thus $\exp(a_i),a_i\in K(\exp(f))$.
Now assume that $f_{\prec}\neq 0$ so $f_\succ$ is a proper truncation of $f$. We have
\[\exp(a_0 + \cdots+a_n))\in \supp_{\exp}(\exp(f))\subseteq \exp(f_\succ)\supp(f_\prec)^*.\]
Note that $\exp(f_\succ)\in K$ and $f_\prec\in K$ by assumption, so $\exp(a_0+\cdots+a_n)\in K$. Now use $\il$-closedness of $K$ to finally conclude that $\exp(a_i),a_i\in K$. Thus $K(\exp(f))$ is $\til$-closed.
\end{proof}

\noindent
Given a subfield $K$ of $\mathbb{T}$ we define the $\exp$-\textit{extension} of $K$ to be the smallest subfield, $K(\exp(K))$, of $\mathbb{T}$ containing $K$, and $\exp(K)$, where 
\[\exp(K)=\{\exp(f): f\in K\}.\]
\noindent
Note that if $K$ is truncation closed and $f\in K$, then $f_\succ = f|_{1}$, and $\exp(f_\succ)\in \exp(K)$.
Thus the $\exp$-extension of $K$ is the $\mathcal{F}$-extension of $K(\{\exp(f_\succ)\}_{f\in K})$ with $\mathcal{F} = \mathcal{F}_1 = \{\sum_n X^n/n!\}$.
Hence we get the following.

\begin{corollary}\label{tL19}
Suppose $K$ is a $\til$-closed subfield of $\T$ containing $\mathbb{R}$. Then $K(\exp(K))$ is $\til$-closed.
\end{corollary}

\noindent
We can define the $\exp$-\textbf{closure} of a field $K$, to be the smallest subfield $L$ of $\T$ containing $K$ such that for any $f\in L$ we have $\exp(f)\in L$. By realizing the $\exp$-closure as a directed union of $\exp$-extensions, we get the following.

\begin{corollary}\label{tL20}
Let $K$ be a $\til$-closed subfield of $L_n$ containing $\R$. Then the $\exp$-closure of $K$ is $\til$-closed.
\end{corollary}

\subsection*{The Liouville Closure} Let $H$ be a differential subfield of $\T$ containing $\R$. For our purposes the \textbf{Liouville closure} of $H$ (inside $\T$) is the (unique) smallest differential subfield of $\T$ that contains $H$ and is Liouville closed. 
As we mentioned in the introduction a differential subfield $H$ of $\T$ is Liouville closed if:
\begin{enumerate}
	\item[(LC1)] $H$ is real closed,
    \item[(LC2)] $\exp(H)\subseteq H$ and,
    \item[(LC3)] for every $g\in H$ there is $g\in H$ such that $g'=h$.  
\end{enumerate}
We call (LC3) being closed under integrals, or antiderivatives. Recall every element of $\Texp$ with constant term $0$ has an antiderivative inside $\Texp$. Moreover since $\int 1 = \ell_1$ then every element of $\Texp$ has an antiderivative in $\Texp(\ell_1)\subseteq L_1$. The following diagram commutes 

\xymatrix{
& & & & & L_n\ar[rr]^-{\int} \ar[d]_{\upper^n} & & L_{n+1}\\
& & & & & \Texp \ar[r]^{\times \mathfrak{e}_n} & \Texp \ar[r]^{\int}& L_1\ar[u]_{\ \downer^n}
}
\noindent
where $\mathfrak{e}_n:= e_0e_1\cdots e_{n-1}$, $e_0 = x$ and $e_{n+1}= \exp(e_n)$.
The above diagram follows from the commutative diagram

\xymatrix{
& & & & & L_n\ar[rr]^-{\partial} \ar[d]_{\upper^n} & & L_n\\
& & & & & \mathbb{T}^E \ar[r]^{\partial} & \mathbb{T}^E\ar[r]^{\div \mathfrak{e}_n} & \mathbb{T}^E \ar[u]_{\ \downer^n}
}
\noindent
which represents the way the derivative is defined.
\begin{lemma}\label{tL21}
Let $K$ be a $\til$-closed differential subfield of $\mathbb{T}$ containing $\R$. Then there is $F$, a $\til$-closed differential field extension of $K$ inside the Liouville closure of $K$ such that every element of $K$ has an antiderivative inside $F$. 
\end{lemma}
\begin{proof}
We may assume that $K$ contains $\ell_n$ for each $n$ by corollary \ref{logs1} and \ref{tL13}. Let $K_n$ be the image under $\downer^n$ of $F_\infty^n$, the smallest differential subfield of $\Texp$ containing $K\upper^n\cap \Texp$ such that for every $m$, $F_m:=F_\infty^n\cap E_m$ is closed under $(I-a\partial)\inv$ for all $a\in F_m$ with $a\prec E_{n-1}$. Then $K_\infty:= \bigcup_n K_n$ is our desired extension. To show this let $f\in K\cap L_n$ be such that $\mathfrak{e}_n(f\upper^n )= \sum_{a\in A_n}f_a\exp(a)$ for $f_a\in E_{n-1}$. since $K$ is $\til$-closed then $f_a\downer^n$ is in $K$ for each $a\in A_n$, and thus $g:=\sum_{a\in A_n^{\neq}}(I-a\partial)\inv(f_a)\exp(a)\in F_\infty$. Let $b$ be the constant term of $f$ then the antiderivatives of $f$ have the form $g\downer b\ell_{n+1} +r$ where $r\in \R$. Since $K_n$ is $\til$-closed for every $n$, then $K_\infty$ is $\til$-closed.         
\end{proof}

We are ready to prove Theorem 1.2
\begin{theorem} Let $K$ be a $\til$-closed differential subfield of $\T$ containing $\R$. Then the Liouville closure of $K$ is also $\til$-closed.
\end{theorem}
\begin{proof}
 Let $L$ be the Liouville closure of $K$ inside $\T$. Using Zorn's lemma we fix $F$ a maximal $\til$-closed differential subfield of $\T$ containing $\R$. By \ref{tL20}  $F$ is exponentially closed, by \ref{tL15} $F$ is real closed, and by Lemma \ref{tL21} $F$ is closed under integrals and thus equal to the Liouville closure of $K$.
\end{proof}

\section{Remarks and Comments}
The proof of Lemma \ref{tL16} is in taken mostly from \cite{D} where a version of it appears with a slightly weaker conclusion. There are notions of composition for generalized power series see for example \cite{H}. One may ask what are (if any) the truncation preservation results that can be obtained under closing certain truncation closed subsets of Hahn fields under composition. Another question that arises from our work is what is a natural way of extending our results to the field of surreal numbers with a derivation, such as the one suggested by Berarducci and Mantova \cite{BM}. Further more we still want to determine if it is possible to find truncation preservation results for extensions given just by antiderivatives of elements of a truncation closed (differential) field.



\begin{thebibliography}{11}

\bibitem{AD} M. Aschenbrenner, L. van den Dries, {\it {$H$}-fields and their Liouville extensions}, Mathematische Zeitschrift, {\bf 242}, (2002),pp. 543--588

\bibitem{ADH} M. Aschenbrenner, L. van den Dries, J. van der Hoeven, {\it Asymptotic Differential Algebra and Model Theory of Transseries}, arXiv:1509.02588v1.


\bibitem{BM} A. Berarducci, V. Mantova, \textit{Surreal numbers, derivations and transseries}  arXiv:1503.00315v3.

\bibitem{D} L. van den Dries, {\it Truncation in Hahn fields,} in: A. Campillo et al. (eds.), {\em Valuation Theory in Interaction}, pp. 578--595, European Mathematical Society, 2014.

\bibitem{DMM} L. van den Dries, A. Macintyre, and D. Marker, {\it Logarithmic-exponential series,} Annals of Pure and Applied Logic {\bf 111} (2001), 61--113. 

\bibitem{F} A. Fornasiero, {\it Embedding Henselian fields into power series.}  J. Algebra 304 (2006), no. 1, 112–156. 

\bibitem{FKK} A. Fornasiero, F-V. Kuhlmann, S. Kuhlmann \textit{Towers of complements to valuation rings and truncation closed embeddings of valued fields}. Journal of Algebra 323 (2010), 574-600.

\bibitem{G} H. Gonshor, \textit{An Introduction to the Theory of Surreal Numbers}, Cambridge university Press, 1986

\bibitem{H} J. van der Hoeven, \textit{Operators on Generalized Power Series}, Illinois Journal of Math 45 (2001) 1161--1190.

\bibitem{MR} M.H. Mourgues, and J.P. Ressayre,{\it Every Real Closed Field Has an Integer Part.} The Journal of Symbolic Logic 58, no. 2 (1993): 641-47.

\bibitem{N} B.H. Neumann \textit{On Ordered Division Rings}, Transactions of the American Mathematical Society, 1949, pp 202--252

\end{thebibliography}
\end{document}